\documentclass[11pt,reqno]{amsart}

\usepackage{graphicx}
\usepackage{geometry}
\usepackage[all]{xy}
\usepackage{tikz}
\usepackage{tikz-cd}
\usetikzlibrary{decorations.pathreplacing}
\usetikzlibrary{arrows}

\usepackage{calrsfs}
\usepackage[T1]{fontenc} 
\usepackage{textcomp}
\usepackage{times}
\usepackage[scaled=0.90]{helvet} 

\usepackage{amsfonts}

\usepackage{amssymb}
\usepackage{hyperref}
\usepackage{mathtools}

\usepackage{etoolbox}
\patchcmd{\section}{\scshape}{\bfseries}{}{}
\makeatletter
\renewcommand{\@secnumfont}{\bfseries}
\makeatother

\DeclareRobustCommand{\SkipTocEntry}[5]{}

\newcommand*{\qedex}{\hfill\ensuremath{\lozenge}}

\newtheorem{introtheorem}{Theorem}

\theoremstyle{definition}
\newtheorem{introexample}[introtheorem]{Example}
\newtheorem*{introexamplecont}{Example B (continued)}
\theoremstyle{plain}
\newtheorem{theorem}{Theorem}[section]
\newtheorem{proposition}[theorem]{Proposition}
\newtheorem{lemma}[theorem]{Lemma}
\newtheorem{corollary}[theorem]{Corollary}
\theoremstyle{definition}
\newtheorem{definition}[theorem]{Definition}
\newtheorem{example}[theorem]{Example}
\newtheorem{remark}[theorem]{Remark}
\newtheorem{question}[theorem]{Question}
\newtheorem{nd}[theorem]{Notations/Definitions}

\def\cprime{'}
\renewcommand{\bar}{\overline}
\renewcommand{\hat}{\widehat}
\renewcommand{\geq}{\geqslant}
\renewcommand{\leq}{\leqslant}
\renewcommand{\Re}{\mathsf{Re}}
\renewcommand{\Im}{\mathsf{Im}}
\DeclareMathOperator{\Z}{\mathbf{Z}}
\DeclareMathOperator{\N}{\mathbf{N}}
\DeclareMathOperator{\Q}{\mathbf{Q}}
\DeclareMathOperator{\R}{\mathbf{R}}
\DeclareMathOperator{\C}{\mathbf{C}}
\DeclareMathOperator{\F}{\mathbf{F}}
\DeclareMathOperator{\PP}{\mathbf{P}}
\DeclareMathOperator{\End}{\mathrm{End}}
\DeclareMathOperator{\im}{\mathrm{im}}
\DeclareMathOperator{\Fr}{\mathrm{Fr}}
\DeclareMathOperator{\ord}{\mathrm{ord}}
\newcommand{\acts}{\,\rotatebox[origin=c]{-90}{$\circlearrowright$}\,}

\begin{document}

\date{1 March 2018 (version 2.0)} 
\title{Dynamics on abelian varieties in positive characteristic}
\author[J.~Byszewski]{Jakub Byszewski}
\address{\normalfont Wydzia\l{} Matematyki i Informatyki Uniwersytetu Jagiello\'nskiego, ul.\ S.\ \L ojasiewicza 6,
30-348 Krak\'ow, Polska}
\email{jakub.byszewski@uj.edu.pl}
\author[G.~Cornelissen]{Gunther Cornelissen}
\address{\normalfont Mathematisch Instituut, Universiteit Utrecht, Postbus 80.010, 3508 TA Utrecht, Nederland} 
\email{g.cornelissen@uu.nl}

\thanks{We thank Fryderyk Falniowski, Marc Houben, Jakub Konieczny, Dominik Kwietniak, Frans Oort, Ze\'ev Rudnick and Tom Ward for feedback on previous versions, Bartosz Naskr\k{e}cki and Jeroen Sijsling for pointing us to the LMFDB, Jan-Willem van Ittersum for crucial corrections in Sagemath code, and Damaris Schindler for help with identifying main and error terms in the final section. JB gratefully acknowledges the support of National Science Center, Poland under grant no.\ 2016/\-23/\-D/ST1/\-01124.}

\subjclass[2010]{37P55, 14K02, 37C30. (11B37, 11N45, 14G17, 37C25.)}
\keywords{\normalfont Abelian variety, endomorphism, inseparability, group scheme, fixed points, Artin--Mazur zeta function, dynamical zeta function, recurrence sequence, holonomic sequence, natural boundary, functional equation, prime orbit distribution}

\begin{abstract} \noindent We study periodic points for endomorphisms $\sigma$ of abelian varieties $A$ over algebraically closed fields of positive characteristic $p$. We show that the dynamical zeta function $\zeta_\sigma$ of $\sigma$ is either rational or transcendental, the first case happening precisely when $\sigma^n-1$ is a separable isogeny for all $n$. We call this condition \emph{very inseparability} and show it is equivalent to the action of $\sigma$ on the local $p$-torsion group scheme being nilpotent. 

The ``false'' zeta function $D_\sigma$, in which the number of fixed points of $\sigma^n$ is replaced by the degree of $\sigma^n-1$, is always a rational function. Let $1/\Lambda$ denote its largest real pole and assume no other pole or zero has the same absolute value. Then, using a general dichotomy result for power series proven by Royals and Ward in the appendix, we find that $\zeta_\sigma(z)$ has a natural boundary at $|z|=1/\Lambda$ when $\sigma$ is not very inseparable. 

We introduce and study  \emph{tame} dynamics, ignoring orbits whose order is divisible by $p$. We construct a tame zeta function $\zeta^*_{\sigma}$ that is always algebraic, and such that $\zeta_\sigma$ factors into an infinite product of tame zeta functions. We briefly discuss functional equations.  

Finally, we study the length distribution of orbits and tame orbits. Orbits of very inseparable endomorphisms distribute like those of Axiom A systems with entropy $\log \Lambda$, but the orbit length distribution of not very inseparable endomorphisms is more erratic and similar to $S$-integer dynamical systems. We provide an expression for the prime orbit counting function in which the error term displays a power saving depending on the largest real part of a zero of $D_\sigma(\Lambda^{-s})$. 
 \end{abstract}

\maketitle

\vspace*{-1cm}
\addtocontents{toc}{\SkipTocEntry}
\tableofcontents

\section*{Introduction} 
The study of the orbit structure of a dynamical system starts by considering periodic points, which, as advocated by Smale in \cite[Section 1.4]{Smale} and Artin--Mazur \cite{AM}, can be approached by considering \emph{dynamical zeta functions}. More precisely, let $S$ denote a set (typically, a topological space, differentiable manifold, or an algebraic variety), 
let $f \colon S \rightarrow S$ be a map on a set $S$ (typically, a homeomorphism, a diffeomorphism, or a regular map), and denote by $f_n$ the number of fixed points of the $n$-th iterate  $f^{n}=f\circ f \circ \dots \circ f \, (n \mbox{ times})$, i.e., the number of \emph{distinct} solutions in $S$ of the equation $f^{n}(x)=x$. Let us say that $f$  is \emph{confined} if $f_n$ is finite for all $n$, and use the notation $f \acts S$ to indicate that $f$ satisfies this assumption. For such $f$, the basic question is to find patterns in the sequence $(f_n)_{n \geq 1}$: Does it grow in some controlled way? Does it satisfy a recurrence relation, so that finitely many $f_n$ suffice to determine all? These questions are recast in terms of the (full) dynamical zeta function, defined as
$ \zeta_f(z) \coloneqq \exp ( \sum\limits f_n z^n/n ). $ 
Typical questions are: 
\begin{enumerate} 
\item[\textup{(Q1)}] Is $\zeta_f$ (generically) a rational function? (Smale \cite[Problem 4.5]{Smale}); 
\item[\textup{(Q2)}] Is $\zeta_f$ algebraic as soon as it has a nonzero radius of convergence? (Artin and Mazur \cite[Question 2 on p.~84]{AM}). 
\end{enumerate}
Answers to these questions vary widely depending on the situation considered; we quote some results that provide context for our study.
The dynamical zeta function $\zeta_f(z)$ is rational when  $f$ is an endomorphism of a real torus (\cite[Thm.\ 1]{Baake}); $f$ is a
rational function of degree $ \geq 2$ on $\PP^1(\C)$ (Hinkkanen  \cite[Thm.~1]{Hinkkanen}); or $f$ is the Frobenius map on a variety $X$ defined over a finite field $\F_q$, so that $f_n$ is the number of $\F_{q^n}$-rational points on $X$  and $\zeta_f(z)$ is the Weil zeta function of $X$ (Dwork \cite{Dwork} and Grothendieck \cite[Cor.~5.2]{Gr}).
Our original starting point for this work was Andrew Bridy's automaton-theoretic proof that $\zeta_f(z)$ is transcendental for separable dynamically affine maps on $\PP^1(\bar \F_p)$, e.g., for the power map $x \mapsto x^m$ where $m$ is coprime to $p$ (\cite[Thm.\ 1]{Bridy}, \cite[Thm.\ 1.2 \& 1.3]{BridyBordeaux}). Finally, we mention that $\zeta_f(z)$ has natural boundary (namely, it does not extend analytically beyond the disk of convergence)  for some explicit automorphisms of  solenoids, e.g.,  the map dual to doubling on $\Z[1/6]$ (Bell, Miles, and Ward \cite{BMW}).

In this paper, we deal with these questions in a rather ``rigid'' algebraic situation, when $S=A(K)$ is the set of $K$-points on an abelian variety over an algebraically closed field of characteristic $p>0$, and $f=\sigma$ is a confined endomorphism $\sigma \in \End(A)$ (reserving the notation $f$ for the general case). 
It is plain that $\zeta_\sigma$ has  nonzero radius of convergence (Proposition \ref{conv}). We provide an \emph{exact dichotomy} for rationality of zeta functions in terms of an arithmetical property of $\sigma \acts A$.  Call $\sigma$ \emph{very inseparable} if  $\sigma^n-1$ is a separable isogeny for all $n \geq 1$. The terminology at first may appear confusing, but notice that the multiplication-by-$m$ map for an integer $m$ is very inseparable precisely when $p{\mid}m$, i.e., when it is an inseparable isogeny or zero. For another example, if $A$ is defined over a finite field, the corresponding (inseparable) Frobenius is very inseparable. 

\begin{introtheorem}[= Theorem \ref{insepz} \& Theorem \ref{sepveryinsep}] \label{B}
Suppose that $\sigma \colon A \rightarrow A$ is a confined  endomorphism of an abelian variety $A$ over an algebraically closed field $K$ of characteristic $p>0$. Then $\sigma$ is very inseparable if and only if it acts nilpotently on the local $p$-torsion subgroup scheme $A[p]^0$. Furthermore, the following dichotomy holds: 
\begin{enumerate}
\item If $\sigma$ is very inseparable, then $(\sigma_n)$ is linear recurrent, and $\zeta_\sigma(z)$ is rational.
\item If $\sigma$ is not very inseparable, then $(\sigma_n)$ is non-holonomic (cf.\ Definition \ref{hol} below), and $\zeta_\sigma(z)$ is transcendental.
\end{enumerate}
\end{introtheorem}

Since the local $p$-torsion group scheme has trivial group of $K$-points, in the given characterisation of very inseparability it is essential to use the scheme structure of $A[p]^0$. When $A$ is ordinary---which happens along a Zariski dense subspace in the moduli space of abelian varieties---very inseparable endomorphisms form a proper ideal in the endomorphism ring. Thus, in relation to question (Q1) above, in our case rationality is \emph{not} generic at all. 

The proofs proceed as follows:  the number $\sigma_n$ is the quotient of the degree of $\sigma^n-1$ by its inseparability degree. We use arithmetical properties of the endomorphism ring of $A$ and the action of its elements on the $p$-divisible subgroup to study the structure of  these degrees as a function of $n$, showing that their $\ell$-valuations are of the form ``(periodic sequence) $\times$ (periodic power of $|n|_\ell$)'' (Propositions \ref{lrab} and \ref{rs}). The emerging picture is that the degree is a very regular function of $n$ essentially controlled by linear algebra/cohomology, but to study the inseparability degree, one needs to use geometry. The crucial tool is a general commutative algebra lemma (Lemma \ref{commalgpowers}). We find that for some positive integers $q, \varpi$,  \begin{equation} \label{intrornsn} \begin{array}{l}  \displaystyle{d_n\coloneqq\deg(\sigma^n-1) = \sum_{i=1}^r m_i \lambda_i^n} \mbox{ for some } m_i \in \Z \mbox{ and distinct } \lambda_i \in \C^*;\mbox{ and} \\  \displaystyle{\deg_i(\sigma^n-1) = r_n |n|_p^{s_n}} \mbox{ for } \varpi\mbox{-periodic sequences } r_n \in \Q^*, s_n \in \Z_{\leq 0}. \end{array} \end{equation} 
Note in particular that this implies that the \emph{degree zeta function} $$D_\sigma(z)\coloneqq\exp ( \sum\limits d_n z^n/n ) = \prod_{i=1}^r (1-\lambda_i z)^{-m_i}$$ (called the ``false zeta function'' by Smale  \cite[p.\ 768]{Smale}) is rational. 
In Theorem \ref{thisisalreadybetter}, we then prove an adaptation of the Hadamard quotient theorem in which one of the series displays such periodic behaviour, but the other is merely assumed holonomic. From this, we can already deduce the rationality or transcendence of $\zeta_\sigma$. In contrast to Bridy's result, we make no reference to the theory of automata. 

\begin{introexample}
We present as a warm up example the case where $E$ is an ordinary elliptic curve over $\F_3$ and let $\sigma = [2]$ be the doubling map and $\tau=[3]$ the tripling map, where everything can be computed explicitly. Although the example lacks some of the features of the general case, we hope this will help the reader to grasp the basic ideas. 
For this example, some facts follow from general theory in Bridy \cite{BridyBordeaux}; and, since $\zeta_\sigma(z)$ equals the dynamical zeta function induced by doubling on the direct product of the circle and the solenoid dual to $\Z[1/6]$ (\cite{BMW}), some properties could be deduced from the existing literature, which we will not do. 

First of all,  $\deg(\sigma^n-1)=(2^n-1)^2 = 4^n - 2 \cdot 2^n + 1$ and $\deg(\tau^n-1) = (3^n-1)^2 = 9^n - 2\cdot  3^n + 1$. The corresponding degree zeta functions are 
$$ D_\sigma(z) = \frac{(1- 2z)^2}{(1-4z)(1-z)} \mbox{ and } D_\tau(z) =  \frac{(1- 3z)^2}{(1-9z)(1-z)}. $$
From the definition, $\sigma$ is not very inseparable but $\tau$ is. In fact, $\tau_n  = \deg(3^n-1)$ and $\zeta_\tau=D_\tau$ but, since we are on an ordinary elliptic curve (where $E[p^m]$ is of order $p^m$), we find 
\begin{align*} & \sigma_n = (2^n-1)^2 |2^n-1|_3 = (2^n-1)^2 r^{-1}_n |n|_3^{-s_n} \\ & \mbox{with } 
 \varpi = 2; r_{2k} = 3, s_{2k} = -1; r_{2k+1} = 1, s_{2k+1} = 0. \end{align*}
In our first proof of transcendence of $\zeta_\sigma(z)$, we use the fact that $\sigma_{2n}$ differs from a linear recurrence by a factor $|n|_3$ to argue that it is not holonomic.  

Since we are on an ordinary curve, the local $3$-torsion group scheme is $E[3]^0 = \mu_3$, which has $\End(E[3]^0) = \F_3$ in which the only nilpotent element is the zero element. Thus, we can detect very inseparability of $\sigma$ or $\tau$ by their image under $\End(E) \rightarrow \End(E[3]^0) = \F_3$ being zero, and indeed, $\tau=[3]$ map to zero, but $\sigma=[2]$ does not. 
\qedex
\end{introexample}

In some cases, we prove a stronger result. Let $\Lambda$ denote a dominant root of the linear recurrence (\ref{intrornsn}) satisfied by $\deg(\sigma^n-1)$, i.e., $\Lambda \in \{ \lambda_i\}$ has $|\Lambda| = \max |\lambda_i|$. In Proposition \ref{lame}, we prove some properties of $\Lambda$, e.g., that $\Lambda>1$ is real and $1/\Lambda$ is a pole of $\zeta_\sigma$.  
\begin{introtheorem}[= Theorem \ref{NBsimple}] \label{introNB} If $\sigma \colon A \rightarrow A$ is a confined, not very inseparable endomorphism of an abelian variety $A$ over an algebraically closed field $K$ of characteristic $p>0$ such that $\Lambda$ is the unique dominant root, then the dynamical zeta function $\zeta_\sigma(z)$ has a natural boundary along $|z|=1/\Lambda$.
\end{introtheorem} 
This result implies non-holonomicity and hence transcendence for such functions; our proof of Theorem \ref{introNB} is independent of that of Theorem \ref{B}. The existence of a natural boundary follows from the fact that the logarithmic derivative of $\zeta_\sigma$ can be expressed through certain ``adelically perturbed'' series that satisfy Mahler-type functional equations in the sense of \cite{MahlerEq}, and hence have accumulating poles (proven in the appendix by Royals and Ward). 
From the theorem we see, in connection with question (Q2) above, that a ``generic'' $\zeta_\sigma$ is far from algebraic (not even holonomic), despite having a positive radius of convergence. 

\begin{introexamplecont}
The dominant roots are $\Lambda_\sigma=4$ and $\Lambda_\tau=9$, which are simple. Since $\zeta_\tau$ is rational, it extends meromorphically to $\C$. We prove that $\zeta_\sigma(z)$ has a natural boundary at $|z|=1/4$, as follows. It suffices to prove this for the function $Z(z)=z \zeta'_\sigma(z)/\zeta_\sigma(z) = \sum \sigma_n z^n$, which we can expand as $$ Z(z) = \sum_{2 {\nmid} n} (2^n-1)^2 z^n + \frac{1}{3} \sum_{2{\mid} n} |n|_3 (2^n-1)^2 z^n; $$ if we write $f(t) = \sum |n|_3 t^n$, then 
$$ Z(z) = \frac{z(1+28z^2+16z^4)}{(1-16z^2)(1-4z^2)(1-z^2)} + \frac{1}{3} \left( f(16z^2)-2f(4z^2)+f(z^2) \right). $$
It suffices to prove that $f(t)$ has a natural boundary at $|t|=1$, and this follows from the fact that $f$ satisfies the functional equation $$ f(z) = \frac{z^2+z}{1-z^3} + \frac{1}{3} f(z^3), $$
and hence acquires singularities at the dense set in the unit circle consisting of all third power roots of unity.
\qedex
\end{introexamplecont}

Section \ref{secvi} constitutes a purely arithmetic geometric study of the notion of very inseparability. We prove that very inseparable isogenies are inseparable and that an isogeny $\sigma\colon E \to E$ of an elliptic curve $E$ is very inseparable if and only if it is inseparable. We give examples where very inseparability is not the same as inseparability even for simple abelian varieties. We study very inseparability using the description of $A[p]^0$ through Dieudonn\'e modules, from which it follows that very inseparable endomorphisms are precisely those of which a power factors through the Frobenius morphism. 

\begin{introexample}
Let $E$ denote an ordinary elliptic curve over a field of characteristic $3$ and set $A=E \times E$; then the map $[2] \times [3]$ is inseparable but not very inseparable, since there exist $n$ for which $2^n-1$ is divisible by $3$. In this case, $\End(A[3]^0)$ is the two-by-two matrix algebra over $\F_3$, which contains non-invertible non-nilpotent elements, and under $\End(A) \rightarrow \End(A[3]^0) = M_2(\F_3)$, $[2] \times [3]$ is mapped to the matrix $\mathrm{diag}(2,0)$, which is such an element. \qedex
\end{introexample}

We then introduce the \emph{tame zeta function} $\zeta^*_{\sigma}$, defined as \begin{equation} \zeta^*_{\sigma}(z) \coloneqq \exp \left( \sum_{p \nmid n} {\sigma_n}{} \frac{z^n}{n} \right), \end{equation} 
summing only over $n$ that are not divisible by $p$. The full zeta function $\zeta_{\sigma}$ is an infinite product of tame zeta functions of $p$-power iterates of $\sigma$ (Proposition \ref{p}). 
Thus, one ``understands'' the full zeta function by understanding those tame zeta functions. Our main result in this direction says that the tame zeta function  belongs to a cyclic extension of the field of rational functions:

\begin{introtheorem}[= Theorem \ref{tamehast}] For any (very inseparable or not) $\sigma \acts A$, a positive integer power of the tame zeta function $\zeta^*_\sigma$ is rational. 
\end{introtheorem} 

The minimal such integral power $t_\sigma>0$ seems to be an interesting arithmetical invariant of $\sigma \acts A$; for example, on an ordinary elliptic curve $E$, one can  choose $t_\sigma$ to be a $p$-th power for  $\sigma \acts E$, but for a certain endomorphism of a supersingular elliptic curve, $t_\sigma=p^2(p+1)$ (cf.\ Proposition \ref{stranget}).

\begin{introexamplecont}
The tame zeta function for $\sigma$ is, by direct computation,  \begin{align*} \zeta_\sigma^*(z) & = \exp \left( \frac{1}{3} \sum_{\substack{3 {\nmid} n \\ 2 {\mid} n}} (2^n-1)^2 \frac{z^n}{n} + \sum_{\substack{3 {\nmid} n \\ 2 {\nmid} n}} (2^n-1)^2 \frac{z^n}{n} \right) \\ & =  \sqrt[9]{\frac{F_2(z)^9 F_{64}(z^6)}{F_8(z^3)^3 F_4(z^2)^3}}, 
\mbox{ where }F_a(z)\coloneqq\frac{(1-az)^2}{(1-a^2z)(1-z)},\end{align*} and hence $t_\sigma=9$. Note that even for the very inseparable $\tau$,  $\zeta^*_\tau(z) = D_\tau(z)/\sqrt[3]{D_{\tau^3}(z^3)}$ is not rational, and $t_\tau=3$. \qedex
\end{introexamplecont}

In Section \ref{secFE}, we investigate functional equations for $\zeta_\sigma$ and $\zeta^*_\sigma$ under $z \mapsto 1/(\deg(\sigma)z)$. For very inseparable $\sigma$, there is such a functional equation (which can also be understood cohomologically), but not for $\zeta_\sigma$ having a natural boundary. On the other hand, we show that all tame zeta functions satisfy a functional equation when continued to their Riemann surface (see Theorem \ref{thmFE}). 

In Section \ref{orbits}, we study the distribution of prime orbits for $\sigma \acts A$. Let $P_\ell$ denote the number of prime orbits of length $\ell$ for $\sigma$. In case of a unique dominant root, we deduce a sharp asymptotics for $P_\ell$ of the form 
\begin{equation} \label{pellintro}
P_{\ell} = \frac{\Lambda^{\ell}  }{\ell r_{\ell} |\ell|_p^{s_{\ell}}} + O(\Lambda^{ \Theta \ell })\mbox{ where }
\Theta \coloneqq \max \{ \Re(s) : D_\sigma(\Lambda^{-s})=0 \}. 
\end{equation}

We average further like in the Prime Number Theorem (PNT). Define the \emph{prime orbit counting function} 
$\pi_\sigma(X)$ and the \emph{tame prime orbit counting function} $\pi^*_\sigma(X)$ by 
$$\pi_\sigma(X)\coloneqq\sum\limits_{\ell \leq X} P_\ell \mbox{ and } \pi^*_\sigma(X)\coloneqq\sum\limits_{\substack{\ell \leq X \\ p {\nmid} \ell}} P_\ell. $$ 
Again, whether or not $\sigma$ is very inseparable is related to the limit behaviour of these functions. 

\begin{introtheorem}[= Theorem \ref{asymptoticsforpnt} and Theorem \ref{PNT-tame}] \label{thmd} If $\sigma \acts A$ has a unique dominant root $\Lambda>1$, then, with $\varpi$ as in \textup{(\ref{intrornsn})} and for $X$ taking integer values, we have:
\begin{enumerate}
\item If $\sigma$ is very inseparable, $ \displaystyle{ \lim_{X \rightarrow +\infty} X \pi_\sigma(X) / \Lambda^X}$ exists and equals  ${\Lambda}/{(\Lambda-1).} $
\item If $\sigma$ is not very inseparable, then $X \pi_\sigma(X) / \Lambda^X $ is bounded away from zero and infinity, its set of accumulation points is a union of a Cantor set and finitely many points (in particular, it is uncountable), and every accumulation point is a limit along a sequence of integers $X$ for which $(X,X)$ converges in the topological group $$\{(a,x)\in {\Z}/{\varpi \Z} \times \Z_p : a \equiv x \bmod{ |\varpi|_p^{-1}}\}. $$
\item For any $k \in \{0,\dots, p\varpi-1\}$, the limit 
$ \lim\limits_{\substack{X \rightarrow +\infty \\ X \equiv k \mathrm{\, mod\, }  p\varpi}} {X \pi^*_\sigma(X)}/{\Lambda^{X}}   =: \rho_k $
exists. 
\end{enumerate}
\end{introtheorem} 

An expression for $\rho_k$ in terms of arithmetic invariants can be found in Formula (\ref{rhok}). We also present an analogue of Mertens' second theorem (Proposition \ref{mertens}) on the asymptotics of $$\mathrm{Mer}(\sigma)\coloneqq\sum_{\ell \leq X} P_{\ell}/\Lambda^\ell$$ in $X$. It turns out that, in contrast to the PNT analogue, such type of averaged asymptotics is insensitive to the endomorphism being very inseparable or not.

\begin{introexamplecont} Including a subscript for $\sigma$ or $\tau$ in the notation, 
M\"obius inversion relates $ P_{\sigma,\ell}$ to the values of $\sigma_\ell$, and hence of $\lambda_i, r_n,s_n$; we find  for the very inseparable $\tau$ that $P_{\tau, \ell} = 9^{\ell}/{\ell} + O(3^{\ell }),$ which we can sum to the analogue of the prime number theorem $ \pi_\tau(X) \sim 9/8 \cdot 9^X/X$. 
The situation is different for the not very inseparable $\sigma$, where 
 \begin{equation} \label{exex} P_{\sigma, \ell}  = \frac{4^{\ell}  }{\ell} \cdot \left\{  \begin{array}{ll} |3 \ell|_3 & \mbox{ if $\ell$ is even} \\ 1  & \mbox{ if  $\ell$ is odd} \end{array} \right\} + O(2^{\ell }),  \end{equation} 
and $\pi_\sigma(X)X/4^X$ has uncountably many limit points in the interval $[1/12,4/3]$ (following the line of thought set out in \cite{Everest-et-al}). 

We find as main term in $\mathrm{Mer}(\tau)$ the $X$-th harmonic number $\sum_{\ell \leq X} 1/\ell$, and, taking into account the constant term from summing error terms in (\ref{pellintro}), we get 
$ \mathrm{Mer}(\tau) \sim \log X +c$ for some $c \in \R$. On the other hand, a more tedious computation gives $\mathrm{Mer}(\sigma) \sim 5/8 \log X + c'$ for some $c' \in \R$. 
\begin{figure}[t]
\begin{center}
\includegraphics[width=8cm]{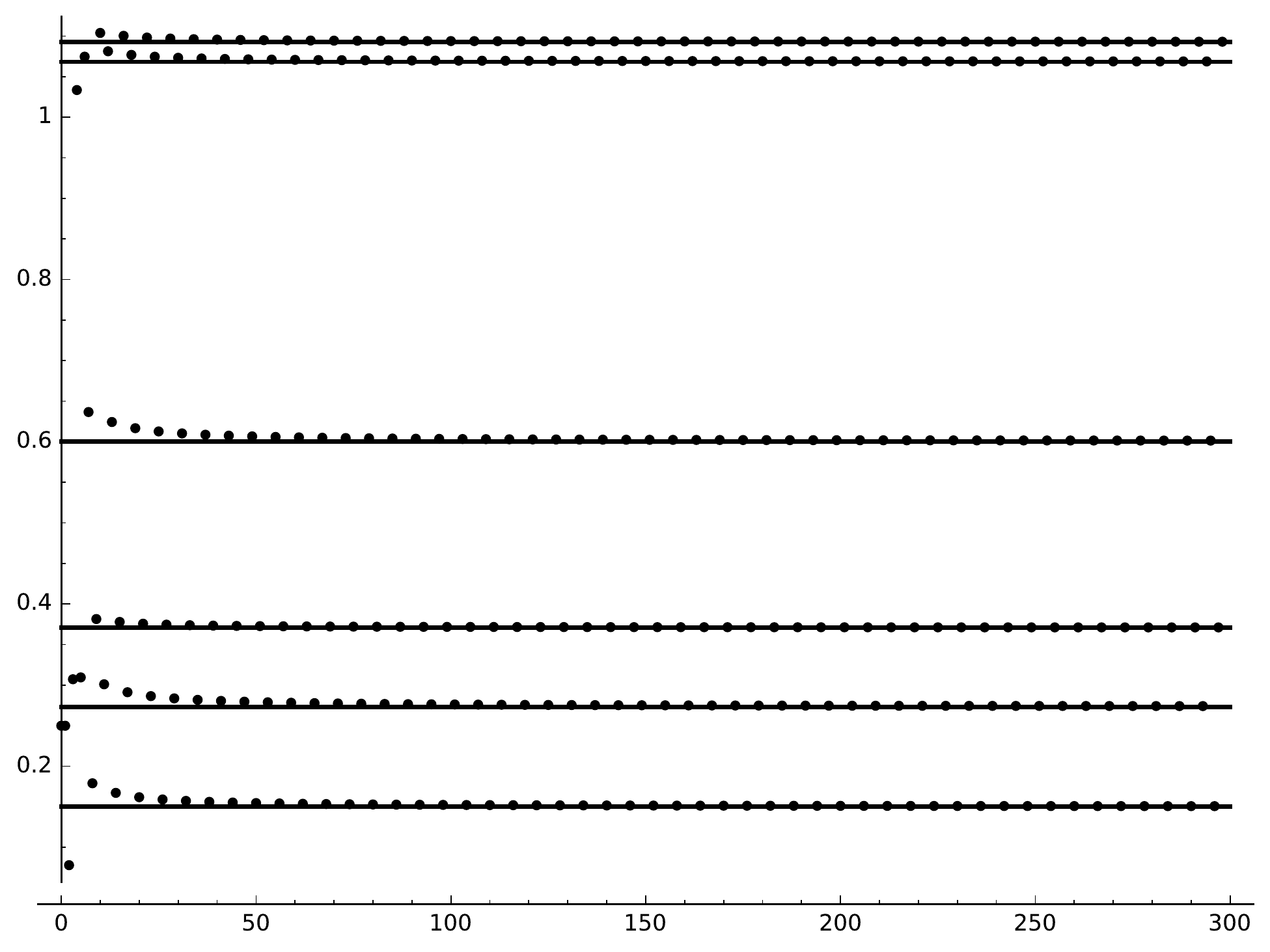}
\caption{Plot of $X \mapsto X \pi^*_\sigma(X)/4^X$, where $\sigma$ is doubling on an ordinary elliptic curve in characteristic 3 (dots) and the six limit values as computed from Formula \eqref{rhok} (horizontal solid lines)}
\label{figord}
\end{center}
\end{figure}

Concerning the tame case, Figure \ref{figord} shows a graph (computed in SageMath  \cite{sagemath}) of the function $\pi_\sigma^*(X)X/4^X$, in which one sees six different accumulation points.  The values $\rho_k$ can be computed in closed form as rational numbers by noticing that if we sum Equation (\ref{exex}) only over $\ell$ not divisible by $3$, we can split it into a finite sum over different values of $\ell$ modulo $6$. We show the computed values in Table \ref{rhotable}, which match the asymptotics in the graph.\footnote{An amusing observation is the similarity between Figure \ref{figord} and the final image in the notorious Fermi--Pasta--Ulam--Tsingou paper  (see the very suggestive Figures 4.3 and 4.5 in the modern account \cite{FPU}): the time averaged fraction of the energy per Fourier mode in the epynomous particle system seems to converge to distinct values, whereas mixing would imply convergence to a unique value; by work of Izrailev--Chirikov the latter seems to happen at higher energy densities. This suggests an analogy (not in any  way mathematically precise) between ``very inseparable'' and ``ergodic/mixing/high energy density''.}\qedex
\begin{table}[t]
\begin{tabular}{c|c|l}
$k$ mod $6$ & $\rho_k \cdot 2^{-2} \cdot 3^3 \cdot 5\cdot  7 \cdot 13$ & $\rho_k$ (numerical) \\
\hline 
$0$ & $839$ & $0.27317867317867$ \\
$1$ & $17 \cdot 193$ & $1.06829466829467$ \\
$2$ & $2^2 \cdot 461$ & $0.60040700040700$\\
$3$ & $461$ & $0.15010175010175$ \\
$4$ & $17 \cdot 67$ & $0.37085877085877$ \\
$5$ & $2^2 \cdot 839$ & $1.09271469271469$ \\
\end{tabular}
\caption{Exact and numerical values of the six limit values in Figure \ref{figord}}
\label{rhotable}
\end{table}
\end{introexamplecont}

We briefly discuss convergence rates in the above theorem (compare, e.g., \cite{Sharp}) in relation to analogues of the Riemann Hypothesis (see Proposition \ref{RH-prop}): there is a function $M(X)$ determined by the combinatorial information $(p,\Lambda,\varpi,(r_n),(s_n))$ associated to $\sigma \acts A$ as in Equation (\ref{intrornsn}), such that for integer values $X$, we have $$\pi_\sigma(X) = M(X) + O(\Lambda^{\Theta X})$$ where the ``power saving'' $\Theta$ is determined by the real part of zeros of the degree zeta function $D_\sigma(\Lambda^{-s})$. Said more colloquially, the main term reflects the growth rate (analogue of entropy) and inseparability, whereas the error term is insensitive to inseparability and determined purely by the action of $\sigma$ on the total cohomology. 

\begin{introexamplecont}
If we collect the main terms using the function, for $k\in\{0,1\}$, $$F_k(\Lambda, X) = \sum_{\substack{ \ell \leq X \\ \ell \equiv k \mathrm{\, mod\, } 2}}{ \Lambda^\ell/\ell}$$ we arrive at the following analogue of the Riemann Hypothesis for $\sigma$:  
$$ \pi_\sigma(X) =  M(X)+O(2^X) \mbox{ with } M(X)\coloneqq \frac{1}{3} F_0(4,X) + F_1(4,X)  -  \sum_{i=1}^{\lfloor \log_3(X) \rfloor}  \frac{2}{9^i} F_{0}\left(4^{3^i},\left\lfloor \frac{X}{3^i} \right\rfloor \right).$$ See Figure \ref{figRH} (computed in SageMath  \cite{sagemath}) for an illustration. \qedex

\begin{figure}[t]
\begin{center}
\includegraphics[width=10cm]{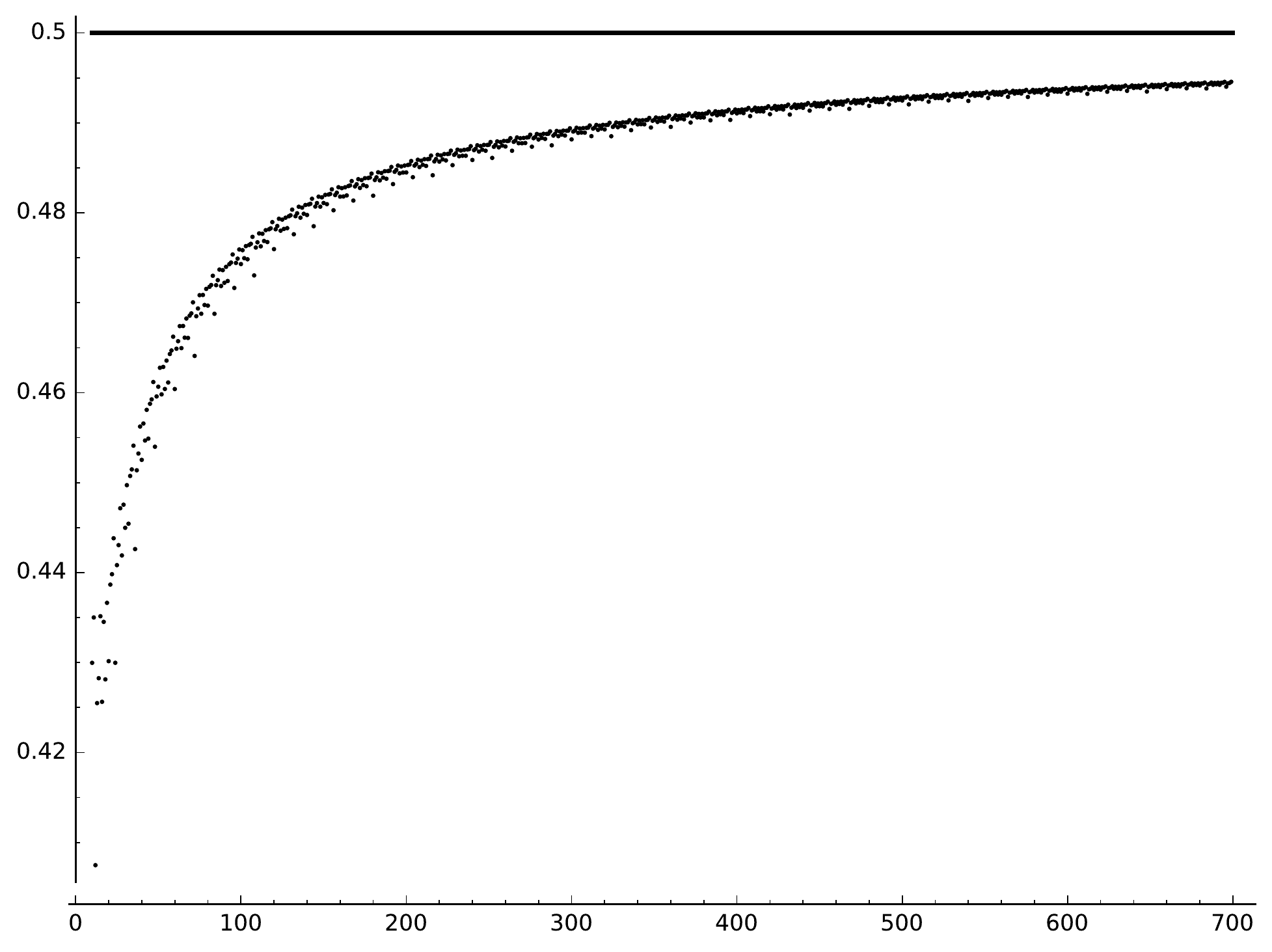}
\caption{Plot of $X \mapsto \log_4 \left| \pi_\sigma(X)-M(X) \right|/X$ (dots) for integer $X \in [10,700]$ and the solid line $\Theta=1/2$, where $\sigma$ is doubling on an ordinary elliptic curve in characteristic 3}
\label{figRH}
\end{center}
\end{figure}
\end{introexamplecont}

\begin{introexample}
All our results apply to the situation where $A$ is an abelian variety defined over a finite field $\F_q$ and $\sigma$ is the Frobenius of $\F_q$, which is very inseparable. This implies known results about curves ${C}/{\F_q}$ when applied to the Jacobian $A=\mathrm{Jac}(C)$ of $C$, such as rationality of the zeta function and analogues of PNT (compare \cite[Thm.\ 5.12]{Rosen}).  
\end{introexample}

We finish this introduction by discussing some open problems and possible future research directions. In the near future, we hope to treat the case of linear algebraic groups, which will require different techniques. Our methods in this paper rest on the presence of a group structure preserved by the map. What happens in absence of a group structure is momentarily unclear to us, but we believe that the study of the tame zeta function in such a more general setup merits consideration. We will consider this for dynamically affine maps on $\PP^1$ in the sense of \cite{BridyBordeaux} (not equal to, but still ``close to'' a group) in future work. It would be interesting to study direct relations between our results and that of compact group endomorphisms and $S$-integer dynamical systems---we briefly touch upon this at the end of Section \ref{caa}.

\section{Generalities} 

\subsection*{Rationality and holonomicity} We start by recalling some basic facts about recurrence sequences. 

\begin{definition} \label{hol} A power series $f=\sum\limits_{n\geq 1} a_n z^n \in \C[\![z]\!]$ is \emph{holonomic} (or \emph{D-finite}) if it satisfies a linear differential equation over $\C(z)$, i.e., if there exist polynomials $q_0,\ldots,q_d\in\C[z]$, not all zero, such that \begin{equation} \label{diffeq} q_0(z) f(z) + q_1(z) f'(z) +\ldots +q_d(z) f^{(d)}(z)=0.\end{equation} A sequence $(a_n)_{n\geq 1}$ is called \emph{holonomic} if its associated generating function $f=\sum\limits_{n\geq 1} a_n z^n \in \C[\![z]\!]$ is holonomic.
\end{definition} 

In the following lemma, we collect some well-known equivalences between properties of a sequence and its generating series: 
\begin{lemma}\label{basicpowerserieslemmata} Let  $(a_n)_{n\geq 1}$ be a sequence of complex numbers.
\begin{enumerate} 
\item\label{basicpowerserieslemmata1}
The following conditions are equivalent: \begin{enumerate} \item The sequence $(a_n)_{n \geq 1}$ satisfies a linear recurrence. \item The power series  $\sum\limits_{n\geq 1} a_n z^n$ is in $\C(z)$. \item There exist complex numbers $\lambda_i$ and polynomials $q_i \in \C[z]$, $1\leq i \leq s$, such that we have $a_n=\sum\limits_{i=1}^s q_i(n)\lambda_i^n \mbox{ for $n$ large enough}.$ 
\end{enumerate}
\item\label{basicpowerserieslemmata5} 
The following conditions are equivalent:  \begin{enumerate} \item The power series  $f(z)=\exp\left(\displaystyle \sum\limits_{n\geq 1} \frac{a_n}{n}  z^n\right)$ is in $\C(z)$. \item There exist integers $m_i$ and complex numbers $\lambda_i$, $1\leq i\leq s$, such that the sequence $a_n$ can be written as $a_n = \sum\limits_{i=1}^s m_i \lambda_i^n$ for all $n\geq 1$.\end{enumerate} Furthermore, if all $a_n$ are in $\Q$, then $f(z)$ is in $\Q(z)$.
\item\label{basicpowerserieslemmata3} 
The following conditions are equivalent: \begin{enumerate} \item The sequence $(a_n)_{n\geq 1}$ is holonomic. \item There exist polynomials $q_0,\ldots,q_d \in \C[z]$, not all zero, such that for all $n\geq 1$ we have $q_0(n)a_n+\ldots+q_d(n)a_{n+d}=0$.\end{enumerate}
Furthermore, if a power series $f(z)\in \C[\![z]\!]$ is algebraic over $\C(z)$, then it is holonomic. 
\end{enumerate}\end{lemma}

\begin{proof}
Statement (\ref{basicpowerserieslemmata1}) follows from \cite[Thm.\ 4.1.1 \& Prop.\ 4.2.2]{StanleyEC1}.  Statement (\ref{basicpowerserieslemmata5}) is \cite[Ex. 4.8]{StanleyEC1}; the final claim holds since $\C(z)\cap \Q(\!(z)\!)=\Q(z)$ (see, e.g., \cite[Lemma 27.9]{MilneLEC}). Statement (\ref{basicpowerserieslemmata3}) is \cite[Thm.\ 1.5 \& 2.1]{Stanley}.
\end{proof} 

\subsection*{Initial reduction from rational maps to confined endomorphisms} Let $A$ denote an abelian variety over an algebraically closed field $K$. Rational maps on abelian varieties are automatically regular \cite[I.3.2]{MilneAV}, and are always compositions of an endomorphism and a translation \cite[I.3.7]{MilneAV}. We say that a regular map $\sigma \colon A \rightarrow A$ is \emph{confined} if the set of fixed points of $\sigma^n$ is finite for all $n$, which we assume from now on. We use the notations from the introduction: $\sigma_n$ is the number of fixed points of $\sigma^n$ and $\zeta_\sigma$ is the Artin--Mazur dynamical zeta function of $\sigma$. 

If $\sigma$ is an endomorphism of $A$, confinedness is equivalent to the finiteness of the kernel $\ker (\sigma^n-1)$ for all $n$, or the fact that all $\sigma^n-1$ are isogenies \cite[I.7.1]{MilneAV}. For arbitrary maps, the following allows us to restrict ourselves to the study of zeta funtions of confined endomorphisms (where case (\ref{id1}) can effectively occur, for example when $\sigma$ is a translation by a non-torsion point):

\begin{proposition} Let  $\sigma \colon A \to A$ be a confined regular map and write $\sigma = \tau_b \psi$, where $\tau_b$ is a translation by $b\in A(K)$ and $\psi$ is an endomorphism of $A$. Then either 
\begin{enumerate}
\item \label{id1} $\sigma_n=0$ for all $n$ and hence $\zeta_{\sigma}(z)=1$; or else 
\item $\psi$ is confined and $\zeta_{\sigma}(z)=\zeta_{\psi}(z)$.
\end{enumerate} \end{proposition}
\begin{proof} Iterates of $\sigma$ are of the form $$\sigma^{n}=\tau_{b^{(n)}}\psi^n,\mbox{ where }b^{(n)}=\sum_{i=0}^{n-1} \psi^i(b).$$ Thus, $\sigma_n=\psi_n$ if $b^{(n)}\in \im(\psi^n-1)$ and $\sigma_n=0$ otherwise. 
If $\sigma_n=0$ for all $n$, then $\zeta_{\sigma}(z)=1$. Otherwise, for some $m\geq 1$ we have $\sigma_m>0$ and thus $b^{(m)} \in \im(\psi^m-1)$, $\sigma_m=\psi_m$, and $\psi^m-1$ is an isogeny. It follows that for all $k\geq 1$ we have $b^{(km)}=\sum_{i=0}^{k-1} \psi^{im}(b^{(m)})$ and hence $b^{(km)} \in \im(\psi^{km}-1)$, $\sigma_{km}=\psi_{km}$, and $\psi^{km}-1$ is an isogeny. Since $\psi^k-1$ is a factor of $\psi^{km}-1$, we conclude that $\psi$ is a confined endomorphism, and hence $\psi^k-1$ is surjective. In particular, $b^{(k)} \in \im(\psi^k-1)$, so $\sigma_n=\psi_n$ for all $n$, and hence $\zeta_{\sigma}(z)=\zeta_{\psi}(z)$.\end{proof}

We make the following standing assumptions from now on, that we will not repeat in formulations of results. Only in Section \ref{secvi} shall we temporarily drop the assumption of confinedness, since this will make exposition smoother (this will be clearly indicated).  
\begin{center}
 \fbox{ \begin{minipage}{0.87 \textwidth} \textbf{Standing assumptions.} \ $K$ is an algebraically closed field of characteristic $p>0$; A is an abelian variety over $K$ of dimension $g$; $\sigma \colon A \rightarrow A$ is a confined endomorphism.
 \end{minipage}}
 \end{center}

\section{Periodic patterns in (in)separability degrees}

For now, we will consider $\zeta_\sigma$ as a \emph{formal power series}
$$ \zeta_\sigma(z) \coloneqq \exp \left( \sum_{n \geq 1} \sigma_n \frac{z^n}{n} \right), $$
and postpone the discussion of complex analytic aspects to  Section \ref{caa}. 
Let $\deg_{\mathrm{i}} (\tau)$ denote the inseparability degree of an isogeny $\tau \in \End(A)$ (a pure $p$-th power). We then have the basic equation \begin{equation} \label{sdi} \sigma_n = \frac{\deg(\sigma^n-1)}{\deg_{\mathrm{i}}(\sigma^n-1)}. \end{equation} 
The strategy is to first consider the ``false'' (in the terminology of Smale \cite{Smale}) zeta function with $\sigma_n$ replaced by the degree of $\sigma^n-1$. This turns out to be a rational function. We then turn to study the inseparability degree, which is determined by the $p$-valuations of the other two sequences. 

We start with a general lemma in commutative algebra that is our crucial tool for controlling the valuations of certain elements of sequences:

\begin{lemma}\label{commalgpowers}
Let $S$ denote a local  ring with maximal ideal $\mathfrak m$ and residue field $k$ of characteristic $p>0$ such that the ring $S/pS$ is artinian. For $\sigma\in S$ and a positive integer $n$,  let $I_n\coloneqq(\sigma^n-1)S$. Let $\overline\sigma$ denote the image of $\sigma$ in $k$. 
\begin{enumerate}
\item\label{commalgpowersi} If $\sigma \in \mathfrak m$, then $I_n=S$ for all $n$. 
\item\label{commalgpowersii} If $\sigma \in S^*$, let $e$ be the order of $\bar \sigma$ in $k^*$. Then: \begin{enumerate}
\item[\textup{(a)}] if $e {\nmid} n$, then $I_n = S$ (this happens in particular if $e=\infty$);

\item[\textup{(b)}] if $e {\mid} n$ and $p {\nmid} m$, then $I_{mn} = I_n$;
\item[\textup{(c)}] there exists an integer $n_0$ such that for all $n$ with $e {\mid} n$ and $\ord_p(n)>n_0$, we have $I_{pn} = p I_n$. 
\end{enumerate}
\end{enumerate}
\end{lemma}
\begin{proof} Part \eqref{commalgpowersi} is clear, so assume $\sigma \in S^*$. 
If $e {\nmid} n$, then $\sigma^n-1$ is invertible in $S$, since $\bar \sigma^n-1 \neq 0$ in $k$ and hence $I_n=S$.

If $e {\mid} n$, we can assume without loss of generality that $e=1$ (replacing $\sigma$ by $\sigma^e$).
Write $\sigma^n = 1 +\varepsilon$ for $\varepsilon \in \mathfrak m$. Then for $m$ coprime to $p$, we immediately find 
\begin{equation*} \sigma^{mn}-1 =  \varepsilon u  \end{equation*}
for a unit $u \in S^*$, and hence $I_{mn}=I_n$, which proves (b). 
On the other hand, 
\begin{equation} \label{Kocyk} \sigma^{pn}-1 = p \varepsilon v + \varepsilon^p  \end{equation}
for some unit $v \in S^*$. This shows that
 $\sigma^{pn}-1 = \varepsilon (pv+\varepsilon^{p-1}) \subseteq \varepsilon \mathfrak m$, which already implies that  we get 
\begin{equation} \label{Przewalski} I_{pn} \subseteq I_n \mathfrak m \mbox{ for all } n. \end{equation}
Since $S/pS$ is artinian, there exists an integer $n_0$ such that $\mathfrak m^{n_0} \subseteq pS$. By iterating (\ref{Przewalski}) $n_0+1$ times, we have 
$$ I_n \subseteq p \mathfrak m \mbox{ for all $n$ with } \mathrm{ord}_{p} (n) > n_0.$$
Assuming now that $\mathrm{ord}_{p} (n) > n_0$, we have $\varepsilon \in p \mathfrak m$, so $\varepsilon^p \in p \varepsilon \mathfrak m$. Hence we conclude from (\ref{Kocyk}) that $\sigma^{pn}-1 = p \varepsilon w$ for some unit $w \in S^*$, and hence $I_{pn} = p I_n$. 
\end{proof}

\subsection*{The degree zeta function} We start by considering the following zeta function with $\sigma_n$ replaced by the degree of $\sigma^n-1$. 

\begin{definition} \label{dzf} The \emph{degree zeta function} is defined as the formal power series $$ D_\sigma(z)\coloneqq \exp\left(\sum_{n \geq 1}\frac{\deg (\sigma^n-1)}{n} z^n\right).$$
\end{definition}

\begin{proposition} \label{lrab} \mbox{ } \begin{enumerate} \item\label{lrab.i} $D_\sigma(z) \in \Q(z)$.
\item\label{lrab.ii} Let $\ell$ be a prime (which might or might not be equal to $p$). Then the sequence of $\ell$-adic valuations $(\vert\deg(\sigma^n-1)|_{\ell})_{n \geq 1}$ is of the form  $$|\deg(\sigma^n-1)|_{\ell}=r_n \cdot |n|_{\ell}^{s_n}$$ for some periodic sequences $(r_n)$ and $(s_n)$ with $r_n \in \Q^*$ and $s_n \in \N$.  Furthermore,  there is an integer $\omega$ such that we have $$r_{n}=r_{\gcd(n,\omega)} \quad \mbox{for } \ell{\nmid}n.$$
\end{enumerate}
\end{proposition} 
\begin{proof}

By \cite[Cor.\ 3.6]{Grieve}, the degree of $\sigma$ and the sequence $\deg (\sigma^n-1)$ can be computed as $$\deg \sigma= \prod_{i=1}^k \mathrm{Nrd}_{R_i/{\Q}}(\alpha_i)^{\nu_i}, \quad \deg (\sigma^n-1)= \prod_{i=1}^k \mathrm{Nrd}_{R_i/{\Q}}(\alpha_i^n-1)^{\nu_i},$$ where $R_i$ are finite dimensional simple algebras over $\Q$, $\alpha_i$ are elements of $R_i$, $\mathrm{Nrd}_{R_i/{\Q}}$ is the reduced norm, and $\nu_i$ are positive integers. These formul\ae\ come from replacing the variety $A$ by an isogenous one that is a finite product of simple abelian varieties and applying the well-known results on the structure of endomorphism algebras of simple abelian varieties. 

After tensoring with $\bar{\Q}$, the algebras $R_i$ become isomorphic to a finite product of matrix algebras over $\bar{\Q}$. For matrix algebras the notion of reduced norm coincides with the notion of determinant, and since the determinant of a matrix is equal to the product of its eigenvalues, we obtain formul\ae\ of the form \begin{equation} \label{zepsutyzamek} \deg (\sigma)= \prod_{i=1}^q \xi_i, \quad \deg (\sigma^n-1)= \prod_{i=1}^q (\xi_i^n-1),\end{equation} with 
$\xi_i \in\bar{\Q}$ (with possible repetitions to take care of multiplicities) and $q=2g$ (since $\deg$ is a polynomial function of degree $2g$). Multiplying out the terms in this expression, we finally obtain a formula of the form \begin{equation} \label{formidable} \deg (\sigma^n-1)= \sum_{i=1}^r m_i \lambda_i^n,\end{equation} for some $m_i \in \Z$  and $\lambda_i \in \bar{\Q}$.  
Now (\ref{lrab.i}) follows from \ref{basicpowerserieslemmata}.(\ref{basicpowerserieslemmata5}).

In order to prove (\ref{lrab.ii}), we will use Formula \eqref{zepsutyzamek}. Consider a finite extension $L$ of the field of ${\ell}$-adic numbers $\Q_{\ell}$ obtained by adjoining all $\xi_i$ with $1\leq i \leq q$. There is a unique extension of the valuation $|\cdot|_{\ell}$ to $L$ that we continue to denote by the same symbol. Then we have $$|\deg (\sigma^n-1)|_{\ell}= \prod_{i=1}^q |\xi_i^n-1|_{\ell}.$$ 

We now claim that for $\xi \in L$, we have \begin{equation} \label{pastadozebow}|\xi^n-1|_{\ell} =\left\{ \begin{array}{ll}  |\xi|^n_{\ell} &  \mbox{ if $|\xi|_{\ell}>1$, }\\ r^{\xi}_n |n|_{\ell}^{s_n^{\xi}} & \mbox{ if $|\xi|_{\ell}=1$,} \\ 1 & \mbox{ if $|\xi|_{\ell}<1$,}\end{array}\right.\end{equation} where $(r^{\xi}_n)_{n}$ and $(s^{\xi}_n)_n$ are certain periodic sequences, $r_n^{\xi} \in \R^*$, $s_n^{\xi} \in \{0,1\}$. 
The first and the last line of the claim are immediate, and the second one follows from applying Lemma \ref{commalgpowers} to the ring of integers $S=\mathcal O_L$ with $\sigma=\xi$, as follows: set $a_n = |\xi^n-1|^{-1}_{\ell}$ and let $e_{\xi}$ be the order of $\xi$ in the residue field of $S$ (note that $e_{\xi}$ is not divisible by $\ell$). Then by Lemma \ref{commalgpowers} there exists an integer $N$ such that $a_n=1$ if $e_{\xi} {\nmid} n$; $a_{mn}=a_n$ if $e_{\xi} {\mid} n$ and $\ell {\nmid} m$; and $a_{\ell n} = \ell a_n$ if $e_{\xi} {\mid} n$ and $\ord_{\ell}(n) \geq N$. Therefore, it suffices to set $(r^\xi_n,s^\xi_n)=(1,0)$ for $e_{\xi} {\nmid} n$; $(r^\xi_n,s^\xi_n)=(a^{-1}_{e_{\xi} l^{\nu}},0)$ for $e_{\xi}{\mid} n$ and $\nu\coloneqq\ord_{\ell}(n)<N$; and $(r^\xi_n,s^\xi_n)=(a^{-1}_{e_{\xi} \ell^N} \ell^{N},1)$ for $e_{\xi} {\mid} n$ and $\ord_{\ell}(n)\geq N$. Note that for $\ell{\nmid } n$ we have $$r^{\xi}_n=\left\{ \begin{array}{ll}  1 &  \mbox{ if $e_{\xi}{\nmid} n$, }\\ a_{e_{\xi}}^{-1} & \mbox{ if $e_{\xi} {\mid} n$.} \end{array}\right. $$

Multiplying together formul{\ae} \eqref{pastadozebow} for $\xi = \xi_1,\ldots,\xi_q$, we obtain  $$|\deg (\sigma^n-1)|_{\ell}= \rho^n r_n |n|_{\ell}^{s_n},$$
where $$\rho = \prod_{i=1}^q \max(|\xi_i|_{\ell},1) \geq 1$$ and $(r_n)$ and $(s_n)$ are periodic sequences, $r_n \in \R^*$, $s_n\in \N$. We claim that $\rho=1$ (that is, there is no $i$ such that $|\xi_i|_{\ell} >1$). Indeed, we know that $\deg(\sigma^n-1)$ is an integer, and hence $\rho^n r_n |n|_{\ell}^{s_n} \leq 1$ for all $n$. Thus, taking  $n\to \infty$, ${\ell}{\nmid} n$, we get $\rho=1$ and $r_n \in \Q^*$. This finishes the proof of the formula for $|\deg (\sigma^n-1)|_{\ell}$. Furthermore, we have $$r_n =\prod_{ e_{\xi_i}{\mid}n} a_{e_{\xi_i}}^{-1} \quad \mbox{for $\ell{\nmid}n$},$$ and hence the final formula holds with $\omega=\mathrm{lcm}(e_{\xi_1},\ldots,e_{\xi_q})$. \qedhere
\end{proof}

\begin{remark} \label{remcohom}
We present an alternative, cohomological description of the degree zeta function $D_\sigma(z)$.  Fix a prime $\ell \neq p$ and let $\mathrm{H}^i\coloneqq \mathrm{H}^i_{\mathrm{\acute{e}t}}(A, \Q_\ell) = \bigwedge^i (V_{\ell} A)^{\vee}$ denote the $i$-th $\ell$-adic cohomology group of $A$, ($V_{\ell} A = T_\ell A\otimes_{\Z_{\ell}} \Q_{\ell}$, $T_\ell A$ is the Tate module and ${}^\vee$ denotes the dual); then 
\begin{equation} \label{cohom} D_\sigma(z) = \prod_{i=1}^{2g} \det(1-\sigma^* z \vert \mathrm{H}^i)^{(-1)^{i+1}}. \end{equation} 
This follows in the same way as for the Weil zeta function: let $\Gamma_{\sigma^n} \subseteq A \times A$ denote the graph of $\sigma^n$ and $\Delta \subseteq A \times A$ is the diagonal \cite[25.6]{MilneLEC}. The Lefschetz fixed point theorem \cite[25.1]{MilneLEC} implies that 
$$(\Gamma_{\sigma^n} \cdot \Delta) = \sum_{i=0}^{2g} (-1)^i \mathrm{tr}(\sigma^n \vert\mathrm{H}^i). $$
Now $\Gamma_{\sigma^n}$ intersects $\Delta$ precisely along the (finite flat) group torsion group scheme $A[\sigma^n-1]$, and hence the intersection number $(\Gamma_{\sigma^n} \cdot \Delta)$ is the order of this group scheme, which is $\deg(\sigma^n-1)$. 
Then the standard determinant-trace identity \cite[27.5]{MilneLEC} implies the result (\ref{cohom}).

The characteristic polynomial of $\sigma_*$ acting on $\mathrm{H}^1$ has integer coefficients independent of the choice of $\ell$ and its set of roots is precisely the set of algebraic numbers $\xi_i$ from the proof of Proposition \ref{lrab} (with multiplicities), see, e.g., \cite[IV.19, Thm.\ 3 \& 4]{Mumford}. 
\end{remark}

\begin{example}
Suppose $A$ is an abelian variety over a finite field $\F_q$ and $\sigma$ is the $q$-Frobenius. Then $\sigma^n-1$ is separable for all $n$, so $\sigma_n = \deg(\sigma^n-1)$ for all $n$, and $\zeta_\sigma(z) = D_\sigma(z)$ is exactly the Weil zeta function of $A{/}{\F_q}$. Thus, we recover the rationality of that function for abelian varieties; note that this is an ``easy'' case: by cutting $A$ with suitable hyperplanes, we are reduced to the case of (Jacobians of) curves, hence essentially to the Riemann--Roch theorem for global function fields proven by F.K.~Schmidt in 1927.
\end{example}

\subsection*{The inseparability degree} Similarly to Proposition \ref{lrab}, we can control the regularity in the sequence of inseparability degrees, with some more (geometric) work; this is relevant in the light of Formula (\ref{sdi}). We start with a decomposition lemma in commutative algebra:

\begin{lemma}\label{commalgdecomp} Let $R$ be a (commutative) ring  and let $M$ be an $R$-module such that for every $m\in M$ the ring $R/\mathrm{ann}(m)$ is artinian. Let $\mathfrak{m}$ be a maximal ideal of $R$. Then the localisation $M_{\mathfrak{m}}$ is equal to $$M_{\mathfrak{m}}=M[\mathfrak{m}^{\infty}]\coloneqq\{m\in M : \mathfrak{m}^k m=0 \text{ for some } k\geq 1\}$$ and $$M=\bigoplus_{\mathfrak{m}} M_{\mathfrak{m}},$$ the direct sum being taken over all maximal ideals $\mathfrak m$ of $R$.
\end{lemma} 

\begin{proof}
Assume first that the module $M$ is finitely generated, say, with generators  $m_1,\ldots,m_s$. Set $I=\mathrm{ann}(M)$. Then $M$ is of finite length as a surjective image of the module $\bigoplus\limits_{i=1}^s R/\mathrm{ann}(m_i)$ and hence the ring $R/I$ is artinian, since it can be regarded as a submodule of $M^s$ via the  embedding $r\mapsto (rm_1,\ldots,rm_s)$. 
 Therefore, the ideal $I$ is contained in only finitely many maximal ideals $\mathfrak{m}_1,\ldots,\mathfrak{m}_s$ of $R$, and for the remaining maximal ideals $\mathfrak m$ of $R$ we have $M_{\mathfrak m}=0$. The artinian ring $R/I$ decomposes as the product \begin{equation} \label{toomuch} R/I \simeq \prod_{i=1}^s R_{\mathfrak{m}_i}/I R_{\mathfrak{m}_i}.\end{equation} Since $I=\mathrm{ann}(M)$, we have $M\otimes_R R/I  \simeq M$ and $M \otimes_R  R_{\mathfrak{m}_i}/I R_{\mathfrak{m}_i} \simeq M_{\mathfrak{m}_i}$. Thus, tensoring (\ref{toomuch}) with $M$, we obtain an isomorphism $$M \to M_{\mathfrak{m}_1}\oplus\ldots\oplus M_{\mathfrak{m}_s}.$$ Since the modules $M_{\mathfrak{m}_i}$ are also of finite length, we  see that each $M_{\mathfrak{m}_i}$ is annihilated by some power of the maximal ideal $\mathfrak{m}_i$.

We now turn to the case of an arbitrary module $M$. Consider the canonical map $$\Phi\colon M \to \prod\limits_{\mathfrak{m}} M_{\mathfrak{m}},$$ the product being taken over all maximal ideals $\mathfrak m$ of $R$. Restricting $\Phi$ to finitely generated submodules $N\subseteq M$, and using the (already established) claim for finitely generated modules, we conclude that the image of $\Phi$ is in fact contained in $\bigoplus\limits_{\mathfrak{m}} M_{\mathfrak{m}}$ and that the induced map $$\Phi \colon M \to \bigoplus_{\mathfrak{m}} M_{\mathfrak{m}}$$ (that we continue to denote by the same letter) is an isomorphism. For a maximal ideal $\mathfrak{n}$ of $R$, multiplication by elements outside of $\mathfrak{n}$ is bijective on $M_{\mathfrak n}$. Therefore, restricting $\Phi$ to $M[\mathfrak{m}^{\infty}]$ shows that $M[\mathfrak{m}^{\infty}]=M_{\mathfrak m}[\mathfrak{m}^{\infty}]$. Finally, we conclude from the case of finitely generated modules that every element in $M_{\mathfrak{m}}$ is anihilated by some power of the maximal ideal $\mathfrak{m}$. Thus, $M[\mathfrak{m}^{\infty}]=M_{\mathfrak{m}}$. \end{proof}

\begin{proposition} \label{rs}
The inseparability degree of $\sigma^n-1$ satisfies
\begin{equation} \label{defrn} \deg_{\mathrm{i}} (\sigma^n-1) = r_n \cdot |n|_p^{s_n} \end{equation}  for periodic sequences $(r_n)$ and $(s_n)$ with $r_n \in \Q^*$ and $s_n \in \Z$, $s_n\leq 0$. Furthermore, there is an integer $\omega$
 such that we have $$r_{n}=r_{\gcd(n,\omega)} \quad \mbox{for } p{\nmid}n.$$
\end{proposition} 

\begin{proof}
The strategy of the proof is as follows: 
since  $\deg_{\mathrm{i}}(\sigma^n-1)$ is a power of $p$, it is sufficient to compute $|\deg(\sigma^n-1)|_p$ and $|\sigma_n|_p$. The former number has been already computed in Proposition \ref{lrab}.(\ref{lrab.ii}); for the latter, we study the $p$-primary torsion of $A$ as an $R$-module, where, not to have to worry about noncommutative arithmetic, we work with the ring $R=\Z[\sigma] \subseteq \End(A)$. Note that $R$ need not be a Dedekind domain. 
Let $X\coloneqq A(K)_{\mathrm{tor}}$ denote the sub{group}  of torsion {points} of $A(K)$. It has a natural structure of an $R$-module, and as an abelian group is divisible; in fact, 
$$X \simeq \left({\Z}\left[\frac{1}{p^{\infty}}\right]/{\Z}\right)^{f} \oplus \bigoplus_{q\neq p}  \left({\Z}\left[\frac{1}{q^{\infty}}\right]/{\Z}\right)^{2g},$$ where   $f$ is the $p$-rank of $A$, and $${\Z}\left[\frac{1}{q^{\infty}}\right] = \bigcup_{k\geq 1} {\Z}\left[\frac{1}{q^k}\right].$$ As $R$ acts on $X$, the localisation $R_{\mathfrak m}$ acts on $X_{\mathfrak m}$ for each maximal ideal $\mathfrak m$ of $R$. Since $X$ is torsion as an abelian group, the conditions of Lemma \ref{commalgdecomp} are satisfied, and hence we have $X_{\mathfrak m} =X[\mathfrak{m}^{\infty}]$ and $$X = \bigoplus_{\mathfrak m} X_{\mathfrak m},$$ the sum being taken over all maximal ideals $\mathfrak m$ of $R$.
 For an element $\tau \in R$, we have $$X[\tau] =  \bigoplus_{\mathfrak m} X_{\mathfrak m}[\tau].$$ Since $X_{\mathfrak m} =X[\mathfrak{m}^{\infty}]$, for any prime number $q$ we have $X_{\mathfrak m}[q^{\infty}]=0$ if $q\not\in \mathfrak m$ and $X_{\mathfrak m}[q^{\infty}]=X_{\mathfrak m}$ if $q\in \mathfrak m$, and hence we get $$X[q^{\infty}]=\bigoplus_{q\in \mathfrak m} X_{\mathfrak m}.$$ Thus the groups  $X_{\mathfrak m}$ for $q \in \mathfrak m$ are $q$-power torsion. It follows that for $\tau \in R$, $\tau\neq 0$, we can compute \begin{equation} \label{letsdothisintermsofvaluations} |X[\tau]|_q = \prod_{q \in \mathfrak m} |X_{\mathfrak m}[\tau]|_q.\end{equation}

Since $X$ is a divisible abelian group, the groups $X_{\mathfrak m}$, being quotients of $X$, are also divisible. Thus, the surjectivity of $p \colon X_{\mathfrak m} \rightarrow X_{\mathfrak m}$ implies that there is a  short exact sequence 
\begin{equation} \label{superconductors} \begin{tikzcd}
        0 \arrow{r} & X_{\mathfrak m}[p] \arrow{r}& X_{\mathfrak m}[p \tau] \arrow{r}{p}  & X_{\mathfrak m}[\tau] \arrow{r}  & 0.  
    \end{tikzcd}
    \end{equation}

Let $\sigma$ be an element of $R$, let $e_{\mathfrak{m}}$ denote the order of $\bar{\sigma}$ in $(R_{\mathfrak m}/\mathfrak{m} R_{\mathfrak m})^*$ for maximal ideals $\mathfrak m$ of $R$ with $p\in \mathfrak m$ and $\sigma \notin \mathfrak{m}$. Note that $e_{\mathfrak m}$ is then coprime with $p$. Applying (\ref{superconductors}) to $\tau = \sigma^n - 1$ and using Lemma \ref{commalgpowers}, we get  
\begin{equation*} |X_{\mathfrak m}[\sigma^{mn}-1]|_p = \left\{ \begin{array}{ll} 1 & \mbox{ for $\sigma \in \mathfrak m$,}\\ 1 & \mbox{ for $\sigma \notin \mathfrak m$ and $e_{\mathfrak m}{\nmid}mn$,}\\ |X_{\mathfrak m}[\sigma^n -1]|_p &  \mbox{ for $\sigma \notin \mathfrak m$, $p {\nmid} m$ and $e_{\mathfrak m}\vert n$, } \\ |X_{\mathfrak m}[\sigma^n -1]|_p \cdot |X_{\mathfrak m}[p]|_p & \mbox{ for $\sigma \notin \mathfrak m$, $m=p$, $e_{\mathfrak m} \vert n $, and $\ord_p (n)\gg 0 $.} \end{array} \right. \end{equation*}
Arguing in the same way as in the proof of Proposition \ref{lrab}, we conclude that there exist periodic sequences $(r^{\mathfrak m}_n)_n$ and $(s^{\mathfrak m}_n)_n$ with $r^{\mathfrak m}_n \in \Q^*$ and $s^{\mathfrak m}_n \in \N$ such that 
\begin{equation} \label{valuationprimaryparts} |X_{\mathfrak m}[\sigma^n-1]|_p = r^{\mathfrak m}_n |n|_p^{s^{\mathfrak m}_n} \quad \mbox{for $n\geq 1$}.\end{equation} Furthermore, $r^{\mathfrak m}_n=1$ and $s^{\mathfrak m}_n=0$ for all $n$ if $\sigma \in \mathfrak m$, and $$r^{\mathfrak m}_n= r^{\mathfrak m}_{\gcd(n,e_{\mathfrak m})} \quad \mbox{for $\sigma \notin \mathfrak m$ and $p{\nmid}n$}.$$ Applying \eqref{letsdothisintermsofvaluations} to $\tau=\sigma^n-1$ and $q=p$, we get the equality $$|\sigma_n|_p = \prod_{p\in \mathfrak m} |X_{\mathfrak m}[\sigma^n-1]|_p.$$ Taking the product of the Formul{\ae} \eqref{valuationprimaryparts} over all maximal  ideals ${\mathfrak m}$ of $R$ with $p\in{\mathfrak m}$, we obtain periodic sequences $(r'_n)_n$ and $(s'_n)_n$ with $r'_n \in \Q^*$ and $s'_n \in \N$ such that $$|\sigma_n|_p = r'_n |n|_p^{s'_n}$$ and $$r'_n=r'_{\gcd(n,\omega')} \quad \mbox{for $p{\nmid} n$,}$$ where $$\omega'=\mathrm{lcm}\{e_{\mathfrak m}\mid \sigma \notin \mathfrak m\}.$$ Writing $$\deg_{\mathrm{i}}(\sigma^n-1)=\frac{\deg(\sigma^n-1)}{\sigma_n}=\frac{|\sigma_n|_p}{|\deg(\sigma^n-1)|_p}$$ and using Proposition \ref{lrab}.\eqref{lrab.ii}, we get sequences $(r_n)$ and $(s_n)$ satisfying having all stated properties except that it might be that $s_n>0$ for some $n$. However, since $\deg_{\mathrm{i}}(\sigma^n-1)$ is an integer, letting $\varpi$ be the common period of $(r_n)$ and $(s_n)$, we automatically get $s_n\leq 0$ for all $n$ such that the arithmetic sequence $n+\varpi\N$ contains terms divisible by arbitrarily high powers of $p$. For all the remaining $n$ we have $\ord_p(n) < \ord_p(\varpi)$, and thus whenever $s_n >0$, we replace $s_n$ by $0$ and $r_n$ by $r_n |n|_p^{s_n}$, obtaining the claim.\qedhere
\end{proof}

\section{A holonomic version of the Hadamard quotient theorem} \label{sechol}

The next proposition is our basic tool from the theory of recurrent sequences. It bears some resemblance to the Hadamard quotient theorem (which is used in its proof), and to conjectural generalisations of it as proposed by Bellagh and B\'ezivin \cite[``Question'' in Section 1]{BellaghBezivin} (using holonomicity instead of linear recurrence) and Dimitrov \cite[Conjecture in 1.1]{Dimitrov} (using algebraicity instead of linear recurrence). In our special case, the proof relies on the quotient sequence having a specific form. 

\begin{proposition}\label{thisisalreadybetter}  Let $(a_n)_{n \geq 1}$, $(b_n)_{n\geq 1}$,  $(c_n)_{n\geq 1}$ be sequences of nonzero complex numbers such that $$a_n=b_n c_n$$ for all $n$. Assume that:
\begin{enumerate}
\item $(a_n)_{n\geq 1}$ satisfies a linear recurrence;
\item $(b_n)_{n\geq 1}$ is holonomic;
\item  $(c_n)_{n\geq 1}$ is of the form $c_n =r_n |n|_p^{s_n}$ for a prime $p$ and periodic sequences $(r_n)_{n\geq 1}$, $(s_n)_{n\geq 1}$ with $r_n\in\Q^*$, $s_n \in \Z$. 
\end{enumerate}
Then the sequence $(c_n)_{n\geq 1}$ is bounded.\end{proposition}
\begin{proof} Note that $c_n \neq 0$ for all $n$. Since the sequence $(b_n)_{n\geq 1}$ given by $b_n = a_n/c_n$ is holonomic, by Lemma \ref{basicpowerserieslemmata}.(\ref{basicpowerserieslemmata3}) there exist polynomials $q_0,\ldots,q_d \in \C[z]$ such that \begin{equation} \label{dfinrel} q_0(n) \frac{a_n}{c_n}=-\sum_{i=1}^d q_i(n+i) \frac{a_{n+i}}{c_{n+i}} \quad\text{for } n\geq 1.\end{equation} We may further assume that $q_0\neq 0$ (otherwise, replace for $i=1,\ldots,d$ the polynomials $q_i$ by $(z-1)q_i$ and shift the relation by one). Suppose $c_n=r_n |n|_p^{s_n}$ is not bounded and let $\varpi$ be the common period of both $(r_n)$ and $(s_n)$. The unboundedness of $(c_n)_{n\geq1}$ means that there exists an integer $j\geq 1$ with $s_j<0$ such that there are elements in the arithmetic sequence $\{j+\varpi n\mid n\geq 0\}$ which are divisible by an arbitrarily high power of $p$. Fix such $j$ and write $s\coloneqq s_j$. Let $\nu$ be an integer such that $p^\nu>\max(d,\varpi)$ and let $\Pi=\mathrm{lcm}(\varpi,p^{\nu})$. Note that $\ord_p \Pi =\nu$. 
By the assumption on $\{j+\varpi n\mid n\geq 0\}$, there exists an integer $J$ such that $J\equiv j \pmod{\varpi}$ and $J\equiv 0 \pmod{p^{\nu}}$. By the definition of the sequence $(c_n)_{n\geq 1}$, for $n \equiv J \pmod{\Pi}$ the values $c_{n+1},\ldots,c_{n+d}$ are uniquely determined (i.e., do not depend on $n$). Substituting such $n$ to the equation \eqref{dfinrel}, we obtain a formula of the form $$\frac{a'_n}{|n|_p^{s}} = b'_n\quad \text{ for } n\equiv J\pmod{\Pi},$$ where $$a'_n=q_0(n)\frac{a_n}{r_j} \quad \mbox{and} \quad b'_n=-\sum_{i=1}^d q_i(n+i) \frac{a_{n+i}}{c_{n+i}} $$ are linear recurrence sequences along the arithmetic sequence  $n\equiv J \pmod{\Pi}$ (here we use the fact that the values $c_{n+1},\ldots, c_{n+d}$ do not depend on $n$, and that linear recurrence sequences form an algebra). Note that the values of $(a'_n)_{n\geq 1}$ are nonzero for sufficiently large $n$, and hence so are $(b'_n)_{n\geq 1}$. By Lemma \ref{basicpowerserieslemmata}.(\ref{basicpowerserieslemmata1}), a subsequence of a linear recurrence sequence along an arithmetic sequence is a linear recurrence sequence. Since the sequence $$|n|_p^{s} = \frac{a'_n}{b'_n}$$ takes values in a finitely generated ring (namely $\Z[1/p]$), we conclude from the Hadamard quotient theorem (van der Poorten \cite[Th\'eor\`eme]{vdPoorten}, \cite{Rumely}) that the sequence $(|J+\Pi n|_p^{s})_{n \geq 0}$ satisfies a linear recurrence, say  \begin{equation}\label{nowigoforcoffee} \gamma_0 |J+\Pi n|_p^s + \gamma_1 |J+\Pi (n+1)|_p^s +\ldots + \gamma_e |J+\Pi (n+e)|_p^s=0 \quad \text{for } n \text{ large enough},\end{equation} where $\gamma_0,\ldots,\gamma_e \in \C$, $\gamma_0\neq 0$. Let $\mu$ be an integer such that $p^{\mu}>\Pi d$. Since  $\nu= \ord_p(\Pi) \leq \ord_p(J)$, we can find an integer $\Pi'>0$ such that $\Pi\Pi'\equiv -J \pmod{p^{\mu}}$. Then for $n\equiv \Pi' \pmod{p^{\mu-\nu}}$ the values of $$|J+\Pi (n+1)|_p^s,\ldots, |J+\Pi (n+e)|_p^s$$ are independent of $n$ (actually, $|J+\Pi (n+j)|_p^s = p^{- \nu s} |j|_p^s$ for $j=1,\dots,e$), and hence by \eqref{nowigoforcoffee} so is the value of $\gamma_0 |J+\Pi n|_p^s$ for $n$ sufficiently large. Substituting $n=\Pi'+ip^{\mu-\nu}$ with $i=0,\ldots,p-1$, we get a contradiction, since there is  exactly one value of $i$ for which $|J+\Pi(\Pi'+ip^{\mu-\nu})|_p^s<p^{- \mu s}$.
\end{proof}

\section{Rationality properties of dynamical zeta functions}

We prove a general rational/transcendental dichotomy  in terms of the following arithmetical property: 

\begin{definition}
An endomorphism $\sigma \in \End(A)$ is called \emph{very inseparable} if $\sigma^n-1$ is a separable isogeny for all $n$. 
\end{definition}

Note that  the zero map is very inseparable. The notion ``very inseparable'' makes sense for arbitrary (not necessarily confined) endomorphisms, but such very inseparable endomorphisms are then automatically confined. We will study the geometric meaning of very inseparability in greater detail in Section \ref{secvi}; here we content ourselves with discussing the case of elliptic curves.

\begin{example} \label{remE}
In case $A=E$ is an elliptic curve, things simplify greatly (compare \cite[Section 5]{BridyBordeaux}): there exists a (nonarchimedean) absolute value $|\cdot|$ on the ring $\End(E)$ such that $\deg_{\mathrm{i}}(\tau)=|\tau|^{-1}$ for $\tau \in \End(E)$. It is immediate that inseparable isogenies together with the zero map form an ideal in $\End(E)$ and that an inseparable isogeny $\sigma$ (i.e., $|\sigma|<1$) is very inseparable (i.e., $|\sigma^n-1|=1$ for all $n$). Neither of these statements  is true in general for higher dimensional abelian varieties.  
\end{example}

\begin{theorem}\label{insepz} 
\mbox{ } 
\begin{enumerate}
\item \label{insepzrat} If $\sigma$ is very inseparable, then $\zeta_\sigma(z) \in \Q(z)$ is rational.
\item \label{insepztran} If $\sigma$ is not very inseparable, the sequence $(\sigma_n)$ is not holonomic, and $\zeta_\sigma(z)$ is transcendental over $\C(z)$.
\end{enumerate}
\end{theorem}

\begin{proof}
Suppose we are in case (\ref{insepzrat}), so $\sigma^n-1$ is separable for all $n$. Since $\sigma_n = \deg (\sigma^n-1)$, Proposition \ref{lrab}.(\ref{lrab.i}) implies that $\zeta_\sigma(z)$ is a rational function of $z$. 

In case (\ref{insepztran}), set  $a_n=\deg(\sigma^n-1)$, $b_n=\sigma_n$, and $c_n=\deg_{\mathrm{i}} (\sigma^n-1)$. By Proposition \ref{lrab}.(\ref{lrab.i}), $(a_n)$ is linear recurrent. By Proposition \ref{rs}, 
$ c_n = r_n |n|_p^{s_n} $ for periodic $r_n \in \Q^*$ and $s_n \in \Z$. Assume, by contradiction, that $b_n$ is holonomic, i.e., that the sequence $(b_n)$ is holonomic. The sequences $(a_n), (b_n)$, and $(c_n)$ then satisfy all the conditions of Proposition \ref{rs}, and we conclude that the sequence $(c_n)$ 
 is bounded. However, the following  proves that $(c_n)$ is unbounded: 

\begin{lemma} \label{unboundedinsepdeg}  If $\sigma$ is not very inseparable, then the sequence $\deg_{\rm i}(\sigma^n-1)$ is unbounded. \end{lemma} 

\begin{proof} By assumption, there exists $n_0$ for which $\sigma^{n_0}-1$ is inseparable. Write $\sigma^{n_0} = 1+\psi$ with $\psi$ inseparable; then $$\sigma^{n_0 p}-1 = (1+\psi)^p-1 = \psi(\psi^{p-1} + p \chi)$$ for some endomorphism $\chi \colon A \to A$. Since $p$ has identically zero differential, the map $\psi^{p-1} + p \chi$ is inseparable, and hence $$ \deg_{\mathrm{i}} (\sigma^{n_0 p}-1) \geq 1 + \deg_{\mathrm{i}}(\psi) = 1 + \deg_{\mathrm{i}} (\sigma^{n_0}-1),$$ and the result follows by iteration.  
\end{proof} 

 To show the transcendence of $\zeta_\sigma(z)$ over $\C(z)$, suppose it is algebraic. Then so would be $$ z \frac{\zeta'_\sigma(z)}{\ \zeta_\sigma(z)} = z(\log(\zeta_\sigma(z)))' = \sum \sigma_n z^n.$$ This contradicts the fact that $\sigma_n$ is not holonomic. 
\end{proof}

\begin{corollary}
At most one of the functions $$ \zeta_\sigma(z) = \exp \left( \sum_{n \geq 1} \sigma_n \frac{z^n}{n} \right) \mbox{ and } 
\frac{1}{ \zeta_\sigma(z)} = \exp \left( \sum_{n \geq 1} -\sigma_n \frac{z^n}{n} \right) $$
is holonomic.
\end{corollary}

\begin{proof} 
Assume that both these functions are holonomic. Since the class of holonomic functions is closed under taking the derivative and the product \cite[Thm.\ 2.3]{Stanley}, we conclude that $ z \frac{\zeta'_\sigma(z)}{\ \zeta_\sigma(z)} $ is holonomic, contradicting Theorem \ref{insepz}.(\ref{insepztran}). 
\end{proof} 

\begin{remark}
It is not true that the multiplicative inverse of a holonomic function is necessarily holonomic. Harris and Shibuya \cite{HS} proved that this happens precisely if the logarithmic derivative of the function is algebraic. We do not know whether $\zeta_\sigma(z)$ is holonomic for not very inseparable $\sigma$, but Theorem \ref{NBsimple} will show that $\zeta_\sigma(z)$ is not holonomic for a large class of maps. 
\end{remark}

\begin{remark}  If $\sigma$ is not assumed to be confined, we could change the definition of $\sigma_n$ by considering $\sigma_n$ to be the number of fixed points of $\sigma^n$ whenever it is finite, and $0$ otherwise. This is in the spirit of \cite{AM}, where only isolated fixed points of diffeomorphisms of manifolds were considered. In this case, we could still prove a variant of Theorem \ref{insepz} saying that if $\sigma$ is a (not-necessarily confined) endomorphism of $A$ such that there exist $n$ such that $\sigma^n-1$ is an isogeny of arbitrarily high inseparability degree, then $(\sigma_n)$ is not holonomic; one needs to use the fact that (the proof of) Proposition \ref{thisisalreadybetter} holds even if we do not insist that $a_n$, $b_n$ be nonzero and instead demand that $c_n=1$ if $a_n=0$.  Note, however, that without the assumption that $\sigma$ is confined, $\zeta_{\sigma}(z)$ could be an algebraic but not rational function. For example,  let $E$ be a supersingular elliptic curve over a field of characteristic $2$, let $A=E\times E$, and $\sigma=[2] \times [-1]$. Then $$\zeta_\sigma(z) = \frac{1-2z}{1+2z} \sqrt{ \frac{(1+z)(1+4z)}{(1-z)(1-4z)} }.$$
  \end{remark} 

\section{Complex analytic aspects}  \label{caa}

We now turn to questions of convergence and analytic continuation.

\subsection*{Radius of convergence}   From the proof of Proposition \ref{lrab}, we pick up the formula 
\begin{equation} \label{equalityxiilambdai}    \deg (\sigma^n-1)= \prod_{i=1}^q (\xi_i^n-1) = \sum_{i=1}^r m_i \lambda_i^n, \end{equation} 
where we note for future use that $q=2g$, $\prod_{i=1}^q \xi_i =\deg(\sigma)$, and $\lambda_i$ are of the form $\lambda_i =  \prod_{j  \in I} \xi_j$ for some $I\subseteq \{1,\ldots,q\},$ each occurring with sign $(-1)^{|I|}$. Recall that $\{\lambda_i\}$ are called the \emph{roots} of the linear recurrence, and $\lambda_i$ is called a \emph{dominant root} if it is of maximal absolute value amongst the roots. The roots $\{\lambda_i\}$ of the recurrence should not be confused with the roots $\{\xi_i\}$ of the characteristic polynomial of $\sigma$ on $\mathrm{H}^1$ (the dual of the $\ell$-adic Tate module for any choice of $\ell \neq p$). 

The following proposition follows from Formula \eqref{equalityxiilambdai} and the fact that  $\deg (\sigma^n-1)$ takes only positive values.

\begin{proposition} \label{lame} \mbox{ } 
\begin{enumerate}
\item \label{noroot1} The $\xi_i$ are not roots of unity. 
\item \label{dom} The linear recurrent sequence $\deg(\sigma^n-1)$ has a dominant positive real root, denoted $\Lambda$. 

\item \label{lambdaismax} $\Lambda = \displaystyle \prod\limits_{i=1}^q \max \{ |\xi_i|,1\} \geq 1$ is the Mahler measure of the characteristic polynomial of $\sigma$ acting on $\mathrm{H}^1$.
\item \label{lambda>1} $\Lambda=1$ if and only if $\sigma$ is nilpotent.
\item \label{udr} $\deg(\sigma^n-1)$ has a unique dominant  root  if and only if there is no $\xi_i$ with $|\xi_i|=1$. 
\item \label{>1} If $\deg(\sigma^n-1)$ has a unique dominant  root  $\Lambda$, then $\Lambda$ has multiplicity $1$. 
\end{enumerate}
\end{proposition}

\begin{proof}

\eqref{noroot1} This is clear since $\sigma$ is confined. 

(\ref{dom}) If not, then $\deg(\sigma^n-1)$ would be negative infinitely often by a result of Bell and Gerhold \cite[Thm.~2]{BellGerhold}.

\eqref{lambdaismax} Denote temporarily $\tilde{\Lambda}=\prod\limits_{i=1}^q \max \{ |\xi_i|,1\}$. We will prove shortly that $\tilde{\Lambda}=\Lambda$. Formula \eqref{equalityxiilambdai} implies that $\Lambda \leq \tilde{\Lambda}$ and 
$$ a_1(n)\coloneqq \sum_{|\lambda_j|=\tilde{\Lambda}} m_j \lambda_j^n $$ equals 
\begin{equation} \label{Jsum} a_1(n)= (-1)^t P^n \prod_{j \in J} (\xi_j^n-1), \end{equation} 
where $t$ is the number of indices $i$ such that $|\xi_i|<1$, $P\coloneqq\prod_{|\xi_i|>1} \xi_i$,  and $J \subseteq  \{1,\ldots,q\}$ denotes the set of indices $i$ such that $|\xi_i|=1$. Since the right hand side of  Formula \eqref{Jsum} is nonzero, we conclude that $\tilde{\Lambda}=\Lambda$. Finally,  by Remark \ref{remcohom}, $\xi_i$ are the roots of the indicated characteristic polynomial.

\eqref{lambda>1} Since none of the $\xi_i$ is a root of unity, and since the set $\{\xi_i\}$ is closed under Galois conjugation, Kronecker's theorem implies that either some $\xi_i$ has absolute value $|\xi_i|>1$, in which case $\Lambda>1$, or else all $\xi_i$ are $0$. The latter is equivalent to $\sigma$ acting nilpotently on $\mathrm{H}^1$, and hence $\sigma$ is nilpotent since $\End(A)$ embeds into (the opposite ring of) $\End (\mathrm{H}^1)$.

\eqref{udr} From Formula \eqref{Jsum} we immediately get that if $J=\emptyset$, then $\deg(\sigma^n-1)$ has a unique dominant root. Conversely, if $J\neq \emptyset$, then substituting $n=0$ into Formula \eqref{Jsum} gives $\sum m_j = 0$, and hence in the formula there are at least two distinct values of $\lambda_j$ occurring, and the dominant root is not unique.

(\ref{>1}) We have already proved that if there is a unique dominant root, then $J = \emptyset$. Thus we read from Formula \eqref{Jsum} that  the multiplicity of $\Lambda$ is $\pm 1$. Since $\deg(\sigma^n-1)$ takes only positive values, the multiplicity is in fact $1$. 
\end{proof} 

\begin{proposition} \label{conv} 
The radius of convergence of the  power series defining $\zeta_\sigma(z)$ is  $1/\Lambda>0$. 
\end{proposition}

\begin{proof}
Note first that we have a trivial bound $\sigma_n=O(\Lambda^n)$, which implies that the power series $\zeta_{\sigma}(z)$ is majorised by $\exp(\sum_{n\geq 1} C\Lambda^n z^n/n)=(1-\Lambda z)^{-C}$ for some constant $C>0$. Thus the radius of convergence of $\zeta_{\sigma}(z)$ is at least $1/\Lambda$. If $\sigma$ is nilpotent, the maps $\sigma^n-1$ are all invertible, and hence $\sigma_n=1$ and $\zeta_{\sigma}(z)=1/(1-z)$. Assume thus that $\sigma$ is not nilpotent, and hence by Proposition \ref{lame}.\eqref{lambda>1}, $\Lambda>1$.

For the other inequality, we write the linear recurrence sequence $\deg(\sigma^n-1)=\sum_{i=1}^r m_i \lambda_i^n$ as the sum of two linear recurrence sequences $a_1(n)$ and $a_2(n)$, $a_1(n)$ as in Formula (\ref{Jsum}) containing the terms with $\lambda_i$ of absolute value $\tilde{\Lambda}=\Lambda$, and $a_2(n)$ containing the terms where $\lambda_i$ is of strictly smaller absolute value. 

Since all $\xi_j$ with $j\in J$ are algebraic numbers on the unit circle but not roots of unity, a theorem of Gel'fond \cite[Thm.\ 3]{Gelfond} implies that for any $\varepsilon>0$ and $n=n(\varepsilon)$ sufficiently large, $$\prod_{j \in J} \left| \xi_j^n-1 \right|>\Lambda^{-n \varepsilon}$$ and hence 
$|a_1(n)| > \Lambda^{n(1-\varepsilon)}$ for sufficiently large $n$. 
The formula in Proposition \ref{rs}  implies that $\deg_{\rm i}(\sigma^n-1) = O(n^s)$ for some integer $s$, and hence it follows from Formula  (\ref{sdi}) that $\sigma_n>\Lambda^{n(1-2\varepsilon)}$ for sufficiently large $n$. An analogous reasoning as for the upper bound proves that the radius of convergence of $\zeta_{\sigma}(z)$ is at most $1/\Lambda^{1-2\varepsilon}$, implying the claim.
\end{proof}

\begin{remark} The value $\log \Lambda$ describes the growth rate of the number of periodic points and plays the role of entropy as defined in the presence of a topology or a measure. It is the logarithm of the spectral radius of $\sigma$ acting on the total ($\ell$-adic) cohomology of $A$---even in the not very inseparable case---similarly to a result of Friedland's in the context of complex dynamics \cite{Friedland1991}. 
\end{remark} 

\subsection*{The degree zeta function} The degree zeta function  $D_\sigma(z)$ is a rational function, and hence admits a meromorphic continuation to the entire complex plane. Actually, $$D_\sigma(z) = \prod_{i=1}^r (1-\lambda_i z)^{-m_i}, $$
written in terms of the parameters in Equation (\ref{equalityxiilambdai}), immediately  provides the extension. Poles (with multiplicity $m_i$) occur at $1/\lambda_i$ with $m_i>0$; zeros (with multiplicity $m_i$) occur at $1/\lambda_i$ with $m_i<0$. We may describe the behaviour of zeros and poles more precisely.

\begin{proposition}\label{proplambda'} Assume that $\sigma$ is not nilpotent. Let $\Lambda'\coloneqq\max\{|\lambda_i| : |\lambda_i|<\Lambda\}< \Lambda.$ \begin{enumerate}\item\label{proplambda'1} The function $D_{\sigma}(z)$ has a pole at $1/\Lambda$.  \item\label{proplambda'2} The function $D_{\sigma}(z)$ has a zero $z_0$ with $|z_0|=1/\Lambda'$ and is holomorphic in the annulus $1/\Lambda<|z|<1/\Lambda'$. \item\label{proplambda'3} $\Lambda' \geq \sqrt{\Lambda}$.
\end{enumerate}\end{proposition}
\begin{proof}  In order to prove \eqref{proplambda'1}, we need to show that the multiplicity $m$ of $\Lambda$ is positive. If $\Lambda$ is a dominant root, this follows from Proposition \ref{lame}.\eqref{>1}. If $\Lambda$ is not a dominant root and $m<0$, the sequence $\deg(\sigma^n-1)-m\Lambda^n$ is a linear recurrent sequence with positive values and  no dominant positive real root, contradicting \cite[Thm.~2]{BellGerhold}. 

Let us now prove (\ref{proplambda'2}). Let $\rho$ denote the minimal value of $|\xi_i|$ and $|\xi_i|^{-1}$ that is strictly larger than $1$, i.e., \begin{align*} \rho&=\min(\min\{|\xi_i| : |\xi_i|>1\}, \min\{|\xi_i|^{-1} : 0< |\xi_i|<1\});\end{align*} it exists since by Proposition \ref{lame}.\eqref{lambda>1}, $\Lambda>1$. Write the set of indices $ \{1,\ldots,q\} = J^{-}_< \cup J^- \cup J \cup J^+ \cup J^{+}_>, $
where membership $i \in J_\ast^\ast$ is defined by the corresponding condition in the second row of the following table
\[\def\arraystretch{1.2}
\begin{array}{ccccc}  J^-_< & J^- & J & J^+ & J^+_> \\ \hline  |\xi_i|<\rho^{-1} & |\xi_i|=\rho^{-1} & |\xi_i|=1&  |\xi_i|=\rho & |\xi_i|>\rho
\end{array} \]
From Equation \eqref{equalityxiilambdai} we see that there is no $\lambda_j$ with $\Lambda/\rho < |\lambda_j|< \Lambda$ and that the terms $\lambda_j$ with $|\lambda_j|=\Lambda/\rho$ arise as products $\prod_{i \in I} \xi_i$ where $I$ contains $J^+_>$, $I$ is disjoint from $J^-_< $, $I \cap J$ can be anything and \emph{either} $I$ contains all except one $i \in J^+$
\emph{or}  $I$ contains all $i \in J^+$ and exactly one $i \in J^-$. 

Setting as before $P\coloneqq\prod\limits_{i\in J^+\cup J^{+}_>} \xi_i$ and $t = \# (J^{-}_<\cup J^-)$, we get \begin{equation}\label{Jsum'} \sum_{|\lambda_j|=\Lambda/\rho} m_j \lambda_j^n = (-1)^{t-1} P^n \prod_{j \in J} (\xi_j^n-1)\left( \sum_{i\in J^+} \xi_i^{-n} + \sum_{i\in J^-} \xi_i^n\right).\end{equation}
Since the right hand side is not identically zero as a function of $n$, we conclude that $\Lambda'=\Lambda/\rho$. We consider two cases. \begin{description}
\item[\normalfont\emph{Case 1}] $J=\emptyset$. Then by Proposition \ \ref{lame}.\eqref{>1}, $P=\Lambda$ has multiplicity $1$ and hence from Formula \eqref{Jsum} we conclude that $t$ is even. Therefore by Formula \eqref{Jsum'} all $\lambda_i$ with $|\lambda_i|=\Lambda'$ have multiplicity $m_i<0$, and hence correspond to zeros of $D_{\sigma}(z)$. 
\item[\normalfont\emph{Case 2}] $J\neq \emptyset$. Substituting $n=0$ into Formula \eqref{Jsum} shows that the sum of multiplicities $m_i$ of $\lambda_i$ with $|\lambda_i|=\Lambda$ is $0$. By  Formula \eqref{Jsum'}, the same is true  for multiplicities $m_j$ of $\lambda_j$ with $|\lambda_j|=\Lambda'$. Thus there is some $\lambda_i$ with $|\lambda_i|=\Lambda'$ and $m_i<0$.
\end{description}

For the proof of \eqref{proplambda'3}, note that since $\Lambda'=\Lambda/\rho$, 
the stated inequality is equivalent to $\Lambda \geq \rho^2$. Since $\Lambda = \prod\max \{ |\xi_i|,1\}$,  it is enough to prove that there are at least two elements in the (non-empty) set $J^+ \cup J^+_>$. Since $q=2g$ is even, it suffices to prove that both $\# J$ and $t=\#(J^- \cup J^-_<)$ are even. Since $\xi_i$ with $|\xi_i|=1$ occur in complex conjugate pairs, $\#J$ is even, and the corresponding term in \eqref{Jsum} is real positive. In the course of proof of Proposition \ \ref{conv} we have shown that the sum $a_1(n)$ dominates the remaining terms, and hence is positive for large $n$. Hence we find from Formula  \eqref{Jsum} that $P>1$ and $t$ is even. 
\end{proof}

\subsection*{Analytic continuation/natural boundary}

When $\sigma$ is very inseparable, $\zeta_\sigma(z)$ coincides with the degree zeta function $D_{\sigma}(z)$ and hence is a rational function. One may wonder whether a P\'olya--Carlson dichotomy holds for the functions $\zeta_\sigma(z)$, meaning that, when they are not rational as above, they admit a natural boundary as complex function (and hence they are non-holonomic; in this context also called ``transcendentally transcendental''). 

  We confirm this for a large class of such maps, providing at the same time another proof of their transcendence (and even non-holonomicity). The crucial tool is Theorem \ref{appendix:theorem} that 
  Royals and Ward prove in Appendix \ref{ap} of this paper. 
\begin{theorem} \label{NBsimple}
Suppose that $\sigma$ is  not very inseparable and that $\Lambda$ is the unique dominant root. Then the function $\zeta_\sigma(z)$ has the circle $|z|= {1}/{\Lambda}$ as its natural boundary. In particular, $\zeta_\sigma(z)$ is not holonomic. 
\end{theorem}

\begin{proof}
We start by the observation that $\zeta_\sigma(z)$ has the same natural boundary as $Z_\sigma(z)\coloneqq\sum \sigma_n z^n$ if the latter function has natural boundary \cite[Lemma 1]{BMW}. Next, we find an expression 
$$ Z_\sigma(z) = \sum_{i=1}^r m_i \sum_{n \geq 1} r_n^{-1} |n|_p^{-s_n} (\lambda_i z)^n, $$
where $m_i, \lambda_i$ are as in (\ref{formidable}) and $r_n, s_n$ are as in Proposition \ref{rs}. 
We now apply Theorem \ref{appendix:theorem}: in the notation of that theorem, we choose $S$ to be the set of primes containing $p$ and all primes $\ell$ for which $|r_n|_\ell \neq 1$ for some $n$. By periodicity of $(r_n)$, the set $S$ is finite. Let $a_n\coloneqq \deg_{\rm i}(\sigma^n-1)=r_n |n|_p^{s_n}$. Suppose $\varpi$ is  a common period for $(r_n)$ and $(s_n)$. For $\ell \in S$, set $n_{\ell}=\varpi$, $c_{\ell, k} = |r_k|_\ell$; for $\ell \neq p$, set $e_{\ell,k}=0$, and set  $e_{p,k} = -s_k$. Then $|a_n|_{S} = a_n^{-1}$, and hence we can write 
$$Z_\sigma(z) =  \sum_{i=1}^r m_i f(\lambda_i z), $$
where $f$ is the function associated to $(a_n)$ as in Theorem \ref{appendix:theorem}. Since $\sigma$ is not very inseparable, by Remark \ref{unboundedinsepdeg} the sequence $(a_n)$ takes infinitely many values. We find that the term $f(\lambda_i z)$ has a natural boundary along $|z|=\frac{1}{|\lambda_i|}$. If $\Lambda$ is the \emph{unique} $\lambda_i$ of maximal absolute value, then the dense singularities along this circle cannot be cancelled by other terms, and we conclude that $Z_\sigma(z)$ has a natural boundary along $|z|=1/\Lambda$, and the same holds for $\zeta_\sigma(z)$. 
Since a holonomic function has only finitely many singularities (corresponding to the zeros of $q_0(z)$ if the series function satisfies Equation \eqref{diffeq}, compare \cite[Thm.\ 1]{Flajolet}), $\zeta_\sigma(z)$ cannot be holonomic. 
\end{proof} 

\begin{question} Is $|z|= {1}/{\Lambda}$ a natural boundary for $\zeta_\sigma(z)$ for any not very inseparable $\sigma$ (even without the assumption of a unique dominant root)? \end{question} 

\subsection*{Metrisable group endomorphisms with the same zeta function.} Given the analogy between our results and some properties of metrizable group endomorphisms, one may ask for the following more formal relationship: 

\begin{question} Can one associate to an action of $\sigma \acts A$ an endomorphism of a compact metrisable abelian group $\tau \acts G$ with the same Artin--Mazur zeta function, i.e., $\zeta_\sigma = \zeta_\tau$? \end{question} 

The analogue of this question over the complex numbers is trivial, as one may take $G=A(\C)$. The degree zeta function $D_\sigma(z)$ artificially equals the Artin--Mazur zeta function of an endomorphism $\tau$ of a $2g$-dimensional real torus whose matrix has the same characteristic polynomial as that of $\sigma$ acting on $T_\ell(A)$ for any $\ell \neq p$ (e.g., the companion matrix). This implies that for a very inseparable $\sigma \acts A$, indeed, $\zeta_\sigma(z)=\zeta_{\tau}(z)$. 

Even in the not very inseparable case, it is sometimes possible to construct such $\tau \acts G$, like we did for the example in the introduction.

In general, it would be natural to consider the induced action of $\sigma$ on the torsion subgroup $A(K)_{\mathrm{tor}}$ (dual of the total Tate module $\prod T_\ell(A)$). This provides the correct contribution $|\sigma_n|_{\ell}$ at all primes $\ell \neq p$; for such $\ell$, the size of the cokernel of $\sigma^n-1$ acting on $T_\ell(A)$ is precisely $|\sigma_n|^{-1}_\ell$. However, at $\ell=p$, we found no such natural group in general, and it seems that $|\sigma_n|_p$ is genuinely determined by the geometry of the $p$-torsion subgroup scheme.

\section{Geometric characterisation of very inseparable endomorphisms} \label{secvi}

In this section, we analyse the condition of very inseparability from a geometric point of view as well as its relation to inseparability. For this, it is advantageous to \emph{temporarily drop the assumption of confinedness} and consider a general $\sigma \in \End(A)$. 

\subsection*{Elementary properties} We start by listing properties of very inseparability that follow more or less directly from the definition. For this, we first write out a very basic property: 
\begin{lemma} \label{iap} Whether $\sigma \in \End(A)$ is a separable isogeny or not is determined by its action on the finite commutative group scheme $A[p]$, i.e., by its image under the map
$ \End(A) \to \End(A[p]). $
\end{lemma} 

\begin{proof} If two endomorphisms $\sigma,\tau \colon A \to A$ induce the same map on $A[p]$, then $\sigma -\tau $ vanishes on the group scheme $A[p]$, and hence it factors through the map $[p]\colon A \to A$. Thus $\sigma-\tau =p\nu$ for some $\nu \colon A \to A$, and hence the map $ \End(A)/p\End(A) \hookrightarrow \End(A[p])$ is injective. Since an endomorphism $A\to A$ is a separable isogeny if and only if it induces an isomorphism on the tangent space, and since every map of the form $p \nu$ induces the zero map on the tangent space, we conclude that $\sigma$ is a separable isogeny if and only if $\tau$ is a separable isogeny.
\end{proof} 

\begin{proposition} \label{vi-elem} Let $\sigma \in \End(A)$. 
\begin{enumerate}
\item \label{vi-fin} The endomorphism $\sigma$ is very inseparable if and only if $\sigma^n-1$ is a separable isogeny for all $n \leq  p^{4g^2}$. 
\item \label{vi-prod} If $A = A_1 \times A_2$ with $A_1, A_2$ abelian varieties and $\sigma=\sigma_1\times \sigma_2$ is a product morphism with $\sigma_i \in \End(A_i)$, then $\sigma$ is very inseparable if and only if $\sigma_1$ and $\sigma_2$ are both very inseparable. 
\item \label{vi-int} Multiplication $[m] \colon A \rightarrow A$ by an integer $m$ is very inseparable if and only if $m$ is divisible by $p$. 
\item \label{vi-ec} An endomorphism of an elliptic curve is very inseparable if and only if it is either an  inseparable isogeny or zero. 
\item\label{vineqi} If $E$ is an elliptic curve over a field of characteristic $3$, then the isogeny $\sigma\coloneqq [2]\times[3]$ on $A\coloneqq E \times E$ is inseparable but not very inseparable. 
\end{enumerate}
\end{proposition} 

\begin{proof} 
To prove (\ref{vi-fin}), observe that by Lemma \ref{iap}, it suffices to look at the images of $\sigma^n-1$ in the ring $\End(A) / p \End(A)$. Since $\End{A}$ is finite free of rank at most $4g^2$, this ring is finite of cardinality $\leq p^{4g^2}$, and hence the sequence of images of $\sigma^n-1$ is ultimately periodic (i.e., periodic except for a finite number of $n$) with all possible values already occuring for $n \leq p^{4g^2}.$ 

Property (\ref{vi-prod}) is immediate from the definition. 

Since an endomorphism of an abelian variety is a separable isogeny if and only if its differential is surjective, to prove (\ref{vi-int}), observe that  the differential of the multiplication by $m^n-1$ map is still given by multiplication by $m^n-1$ and hence is surjective if and only if it is nonzero, i.e., when $p$ does not divide  $m^n-1$. The latter happens for all $n\geq 1$ if and only if $p {\mid} m$.

Statement (\ref{vi-ec}) was already discussed in Remark \ref{remE}.

Property \eqref{vineqi} follows immediately from \eqref{vi-prod} and \eqref{vi-int}.  
\end{proof} 

\subsection*{Using the local group scheme $A[p]^0$} The category of finite commutative group schemes over $K$ is abelian and decomposes as the product of the category of finite \'etale and the category of finite local group schemes (see, e.g., \cite[A \S 4]{Goren}). The group scheme $A[p]$ decomposes canonically as the product of the \'etale part $A[p]_{\mathrm{\acute{e}t}}$ and the local part $A[p]^0$. We now provide a geometric characterisation of (very) inseparability using the local $p$-torsion subgroup scheme, as in Theorem \ref{B} in the introduction.

\begin{theorem} \label{sepveryinsep} Let $\sigma \in \End(A)$. 
 \begin{enumerate} \item \label{it'sgoingtorain1} $\sigma$ is a separable isogeny if and only if it induces an isomorphism on $A[p]^0$. \item \label{it'sgoingtorain2} $\sigma$ is very inseparable if and only if it induces a nilpotent map on $A[p]^0$.\end{enumerate}
\end{theorem}
\begin{proof}
Under the splitting $A[p]=A[p]_{\mathrm{\acute{e}t}} \times A[p]^0$, the morphism $\sigma[p]$ induced by $\sigma$ on $A[p]$ splits as a product morphism $\sigma[p]= \sigma[p]_{\mathrm{\acute{e}t}} \times \sigma[p]^0$. Therefore, we have \begin{equation} \label{splitet} \ker \sigma[p] = \ker \sigma[p]_{\mathrm{\acute{e}t}} \times \ker\sigma[p]^0.\end{equation} An isogeny $\sigma$ is separable if and only if $\ker \sigma$ is \'etale. 

We turn to the proof of (\ref{it'sgoingtorain1}). In one direction, first assume that $\sigma$ is a separable isogeny. Then $\ker \sigma$ is \'etale, and hence so is its subgroup scheme $\ker \sigma[p]$. From the decomposition (\ref{splitet}), we conclude that $\ker\sigma[p]^0$ is both \'etale and local, hence trivial. Since $A[p]^0$ is a finite group scheme, the map $\sigma[p]^0$ is an isomorphism.

For the other direction, assume first that $\sigma$ is \emph{not an isogeny.} Let $B$ be the reduced connected component of $0$ of $\ker \sigma$. Then $B$ is an abelian subvariety, $B[p]^0$ is a nontrivial group scheme (because multiplication by $p$ on $B$ is not \'etale) and is contained in the kernel of $\sigma[p]^0$ and hence $\sigma[p]^0$ is not an isomorphism. 

Secondly, assume that $\sigma$ is an \emph{inseparable isogeny.} Then $\ker \sigma$ is not \'etale. We have $\ker \sigma \subseteq A[n]$ for $n=\deg \sigma$. Writing $n=p^t u$ with $u$ coprime with $p$, we get a decomposition $\ker \sigma =\ker \sigma [p^t]\times \ker \sigma[u]$. The group scheme $\ker \sigma [u]$ is \'etale (as a subgroup scheme of $A[u]$), and hence $\ker \sigma[p^t]$ cannot be \'etale, which means that $\ker \sigma[p^t]^0$ is nontrivial. For each integer $r$, we have an exact sequence \begin{equation*}  \begin{tikzcd}
        0 \arrow{r} & \ker \sigma [p^{r-1}]^0 \arrow{r}& \ker \sigma [p^{r}]^0\arrow{r}{p^{r-1}} & \ker \sigma [p]^0.  
    \end{tikzcd}
    \end{equation*}
Applying this inductively for $r=t, t-1,\ldots, 2$, we conclude that $\ker \sigma[p]^0$ is nontrivial, and hence the morphism $\sigma[p]^0$ is not an isomorphism. This proves (i). 

For the proof of (ii), consider the natural homomorphism $\varphi\colon \End(A) \to \End(A[p]^0)$. Since $\End(A)$ is a finite $\Z$-algebra, and since $p \in \ker \varphi$, the ring $R\coloneqq\mathrm{im}(\varphi)$ is a finite $\F_p$-algebra. By 
part (i), the map $\sigma^n -1$ is a separable isogeny if and only if its image $\varphi(\sigma^n-1)$ is a unit in $\End(A[p]^0)$. We claim that $\varphi(\sigma^n-1)$ is then a unit in $R$; in fact, the ring $R$ is a finite $\F_p$-algebra, and hence there exists a monic polynomial $f\in \F_p[t]$, $f=t^d+a_{d-1}t^{d-1}+\ldots +a_0$ of lowest degree such that $f(\sigma^n-1)=0$. If the constant term $a_0$ of $f$ is different than zero, then we easily see that $\sigma^n-1$ is invertible in $R$, its inverse being $-a_0^{-1} \sum_{i=0}^{d-1} (\sigma^n-1)^i$.  If on the other hand $a_0 =0$, then $\sigma^n-1$ is a two-sided zero-divisor in $R$, hence in $\End(A[p]^0)$, and therefore cannot be a unit in $\End(A[p]^0)$. Thus, our claim is now reduced to the proof of the following lemma.\let\qed\relax\end{proof}

\begin{lemma} Let $R$ be a finite (not necessarily commutative) $\F_p$-algebra and let $r\in R$. Then the following conditions are equivalent: \begin{enumerate} \item For all positive integers $r^n-1$ is invertible. \item The element $r$ is nilpotent.\end{enumerate}
\end{lemma}
\begin{proof}
Let $J$ denote the Jacobson radical of $R$. The ring $R$ is artinian and hence the ring $\bar{R}=R/{J}$ is semisimple \cite[4.14]{Lam}. For an element $s\in R$, denote the image of $s$ in $\bar{R}$ by $\bar{s}$. Then $s$ is invertible in $R$ if and only if $\bar{s}$ is invertible in $\bar{R}$ \cite[4.18]{Lam} and $s$ is nilpotent if and only if $\bar{s}$ is nilpotent (this follows from the fact that the Jacobson radical of an artinian ring is nilpotent, see \cite[4.12]{Lam}). Thus we have reduced the claim to the case of a semisimple ring $\bar{R}$. 

By the Wedderburn--Artin theorem \cite[3.5]{Lam}, a semisimple ring is a product of matrix rings over division rings which in our case need to be finite, and hence by another theorem of Wedderburn \cite[13.1]{Lam} are commutative. Thus we can decompose the ring $\bar{R}$ as a product of matrix rings over finite fields $$\bar{R} \simeq \prod_{i=1}^s \mathrm{M}_{n_i}(\F_{q_i}).$$ Clearly, each of the properties in the statement of the lemma can be considered separately for each term in this product, and we are reduced to proving that a matrix $N$ over a finite field has the property that $N^n -1$ is invertible for all $n\geq 1$ if and only if $N$ is nilpotent. 

If $N$ is nilpotent, then all the matrices $N^n-1$ are  invertible, since in any ring the sum of a unit and a nilpotent that commute with each other is a unit. Conversely, if $N$ is not nilpotent, then $N$ has some eigenvalue $\lambda \neq 0$, perhaps in a larger (but still finite) field. Let $n\geq 1$ be such that $\lambda^n=1$ (such $n$ always exists  in a \emph{finite} field). Then the matrix $N^n-1$ is not invertible. 
\end{proof}

We have some immediate corollaries (where \ref{follows}.(\ref{iap0}) refines Lemma \ref{iap}): 

\begin{corollary} \label{follows} Let $\sigma \in \End(A)$. 
\begin{enumerate} 
\item \label{iap0} Whether $\sigma$ is a separable isogeny or not, or very inseparable or not, is determined by its action on $A[p]^0$, i.e., on its image under the map
\begin{equation*} \label{into0} \End(A) \rightarrow \End(A[p]^0). \end{equation*} 
\item \label{viisi} Very inseparable isogenies are inseparable. 
\item \label{invis} There exists a simple abelian surface with a confined isogeny that is inseparable but not very inseparable and for which inseparable isogenies together with the zero map do not form an ideal.
\end{enumerate}
\end{corollary} 

\begin{proof} 
Statement (\ref{iap0}) is immediate from Theorem \ref{sepveryinsep}. 
Statement (\ref{viisi}) follows from Theorem \ref{sepveryinsep}, since nilpotents are not invertible.
Concerning (\ref{invis}), the following is an example of a simple abelian variety $A$ and an inseparable but not very inseparable  isogeny $\sigma$ (all computational data used can be found at \cite{lmfdb-curve}). Consider the isogeny class of supersingular abelian surfaces over $\F_5$ of $p$-rank $0$ with characteristic polynomial of the Frobenius $\pi$ equal to $x^4+25=0$. The splitting field $L\coloneqq\Q(\pi)=\Q(i,\sqrt{10})$ has no real embeddings, hence by Waterhouse \cite[Thm.\ 6.1]{Waterhouse} there exists a simple abelian surface $A$ with endomorphism ring ${\mathcal O}_L=\Z[i,\pi]$ (the ring of integers in $L$, containing both $\pi$ and $5/\pi=-i \pi$). 
Consider $\sigma=i-2=\frac{\pi^2}{5} - 2$, with characteristic polynomial $\sigma^2+4\sigma+5=0$. The endomorphism $\sigma$ is a confined isogeny since on a simple abelian variety these are exactly the endomorphisms that are neither zero nor roots of unity. Denoting the reduction of $\sigma$ modulo $5$ by $\bar \sigma$, we find that \begin{equation} \label{red} \bar \sigma^2 = \bar \sigma. \end{equation} Note that $A[p]=A[p]^0$ and hence there is an injective map $\mathcal{O}_L/5\mathcal{O}_L \hookrightarrow \End(A[p]^0)$. Now $\sigma$ is separable if and only if $\bar \sigma$ is an isomorphism on $A[p]^0$, which, by (\ref{red}), happens exactly if $\bar \sigma = 1$. But then $\sigma = 5 \psi + 1$ for some $\psi \in \mathcal{O}_L$, which does not hold. Hence $\sigma$ is inseparable. 
On the other hand, $\sigma$ is very inseparable if and only if $\bar \sigma$ is nilpotent on $A[p]^0$, which, by (\ref{red}), happens exactly if $\bar \sigma = 0$. This means that $\sigma = 5 \psi$ for some $\psi \in \mathcal{O}_L$, which does not hold either. Hence $\sigma$ is not very inseparable. 

Let $\sigma'=-i-2$. We similarly prove that $\sigma'$ is inseparable, and yet the map  $\sigma+\sigma'=-4$ is a separable isogeny. Hence the set of inseparable isogenies together with the zero map is not closed under addition.
\end{proof} 

\subsection*{Using Dieudonn\'e modules} 

The structure of the endomorphism ring of the local group scheme $A[p]^0$ can  be computed explicitly using the theory of Dieudonn\'e modules, and we will use this to deduce some more results on very inseparability. 

The group schemes $A[p]$ and $A[p]^0$ are objects in the category $\mathcal{C}_K$
  of finite commutative group schemes over $K$ annihilated by $p$. By covariant Dieudonn\'e theory \cite[A \S 5]{Goren} there is an equivalence of categories
$$ D \colon \mbox{$\mathcal{C}_K$} \rightarrow \mbox{{\sf Finite length left $\mathbf{E}$-modules}}, $$
where $\mathbf{E}=K[F,V]$ denotes the non-commutative ring of polynomials  with relations \begin{equation*} \label{relE} FV=VF=0, F\lambda=\lambda^p F \mbox{ and }V\lambda^p = \lambda V \mbox{ for }\lambda \in K.\end{equation*} 
We may  consider being a very inseparable endomorphism or a separable isogeny as a property of the image of an endomorphism under the map
$ \End(A) \rightarrow \End_{\mathbf{E}} (D(A[p]^0)). $

\begin{example} \label{ESS} If $A$ is an ordinary elliptic curve, then $A[p]^0 \cong \mu_p$, so $\End(A[p]^0) = \F_p$. If $A$ is a supersingular elliptic curve, the local group scheme $A[p]^0$ 
        is the unique non-split self-dual extension  of $\alpha_p$ by $\alpha_p$. 
The Dieudonn\'e module is $D(A[p]^0)=\mathbf{E}/\mathbf{E}(V+F)$ \cite[A.5.4]{Goren} and 
a computation \cite[A.5.8]{Goren} gives a ring isomorphism  \begin{align*} \label{endap0} \End(A[p]^0) \cong \End_{\mathbf{E}} (\mathbf{E}/\mathbf{E}(V+F)) & \cong 
\left\{\begin{pmatrix} a^p & b\\ 0 & a\end{pmatrix} : a \in \F_{p^2}, b\in K\right\}. \end{align*}       
From these computations, one also sees directly that non-invertible elements are nilpotent in $\End(A[p]^0)$ in both the ordinary and the supersingular case, giving an alternative proof of \ref{vi-elem}.(\ref{vi-ec}). \end{example}        

\begin{proposition} \mbox{ } \label{further} Let $\sigma \in \End(A)$ and set $\mathbf{D}\coloneqq D(A[p])^0)$. 
\begin{enumerate}
\item\label{modV} $\sigma$ is a separable isogeny (respectively, very inseparable endomorphism) if and only if its image in $\End_{K[F]} (\mathbf{D}/V\mathbf{D})$ is invertible  (respectively, nilpotent). 
\item \label{vifrob} $\sigma$ is very inseparable if and only if a power of $\sigma$ factors through the $p$-Frobenius map $\Fr \colon A \mapsto A^{(p)}$. 
\item\label{ideal} If $\End(A)$ is commutative, the set of very inseparable endomorphisms forms an ideal in $\End(A)$.
\item \label{notideal} There exists an abelian variety for which the set of very inseparable endomorphisms  is not closed under either addition or multiplication (in particular, it is not an ideal).
\item \label{ord} Let $A$ denote a simple ordinary abelian variety defined over a finite field $\F_q \subseteq K$ with (commutative) endomorphism ring $\mathcal{O}\coloneqq \End(A)$ and Frobenius endomorphism $\pi$. Set $R\coloneqq\Z[\pi, q/\pi]$. Then $R \subseteq \mathcal{O}$ and if $p {\nmid} [\mathcal{O}{:}R]$, then any isogeny of $A$ is very inseparable if and only if it is inseparable. This is in particular true if $q=p\geq 5$.
\end{enumerate}
\end{proposition} 

\begin{proof}
We first prove (\ref{modV}). The relations in $\mathbf E$ imply that $V \mathbf E$ is a two-sided ideal in $\mathbf E$. In this way, $\sigma$, as an $\mathbf E$-endomorphism of $\mathbf{D}$, gives rise to an endomorphism $\tilde \sigma$ of the $\mathbf E / V \mathbf E = k[F]$-module $\mathbf{D}/V\mathbf{D}$. The first claim is that $\sigma$ is nilpotent if and only if $\tilde \sigma$ is. The interesting direction is where $\tilde \sigma$ is nilpotent, meaning that $\sigma^n(\mathbf{D}) \subseteq V\mathbf{D}$ for some $n$. Since $V$ is nilpotent on $\mathbf{D}$ \cite[A.5]{Goren}, say $V^d\mathbf{D}=0$, we can iterate the equation to get $\sigma^{nd}(\mathbf{D}) \subseteq V^d \mathbf{D} = 0$. Secondly, we claim that $\sigma$ is invertible if and only if $\tilde \sigma$ is so. Again, the interesting direction is when $\tilde \sigma$ is invertible. If we let $\mathbf{D}'$ denote the image of $\sigma \colon \mathbf{D} \rightarrow \mathbf{D}$, then $\mathbf{D}'$ is an $\mathbf E$-submodule of $\mathbf{D}$ and $\mathbf{D} = \mathbf{D}' + V\mathbf{D}$. Iterating this sufficiently many times, we find that 
$$\mathbf{D} = \mathbf{D}' + V\mathbf{D}=  \mathbf{D}' + V\mathbf{D}'+ V^2 \mathbf{D} = \ldots = \mathbf{D}'+V\mathbf{D}'+\dots+V^{d-1}\mathbf{D}' \subseteq \mathbf{D}'.$$ This shows that $\sigma$ is surjective, and, since it is an endomorphism of the underlying finite dimensional vector space, it is then automatically injective. 

In order to prove (\ref{vifrob}), note that the Dieudonn\'e module $D(A^{(p)}[p]^0)$ can be identified with $\mathbf{D}=D(A[p]^0)$ with the $\mathbf{E}$-action twisted by the geometric Frobenius map $\psi\colon K\to K$, $\psi(\lambda)=\lambda^{1/p}$. Under this identification, the map induced by the $p$-Frobenius $\Fr \colon A \to A^{(p)}$ on the Dieudonn\'e modules is the $\psi$-semilinear map $V\colon \mathbf{D} \to \mathbf{D}$ \cite[A.5]{Goren}. Moreover, the map $V$ is nilpotent.

If $\sigma$ is very inseparable, there exists $n$ with $\sigma^n|_{A[p]^0}=0$. Since $A[\Fr] \subseteq A[p]^0$, we have $\sigma^n|_{A[\Fr]}=0$ and hence $\sigma^n$ factors through $\Fr$. Conversely, suppose that $\sigma^n = \tau \circ \Fr$ for some $\tau \colon A^{(p)} \rightarrow A$. Passing to the Dieudonn\'e modules, and using the fact that the map $D(\tau)$ is $\psi^{-1}$-semilinear (and hence commutes with $V$), we see that $D(\sigma^n)\mathbf{D} \subseteq V \mathbf{D}, $ so $D(\sigma)$ is nilpotent modulo $V$.   By part (\ref{modV}), we find that $\sigma$ is very inseparable. 

For the proof of \eqref{ideal}, note that, without any assumptions on the ring $\End(A)$, the set $I$ of maps in $\End(A)$ that factor through the $p$-Frobenius $\Fr$ is a left ideal in $\End(A)$. Therefore by \eqref{vifrob}, if the ring $\End(A)$ is commutative, the set of very inseparable maps in $\End(A)$ coincides with the radical  of $I$, and hence is an ideal. 

For (\ref{notideal}), consider $A=E \times E$ for an ordinary elliptic curve $E$. Then $\End(A) = \mathrm{M}_2(\End(E))$ surjects onto $\End(A[p]^0) = \mathrm{M}_2(\F_p)$ (see Example \ref{ESS}). The set of very inseparable endomorphisms corresponds under this map to matrices whose image in $\mathrm{M}_2(\F_p)$ is  nilpotent,  and it suffices to remark that the set of nilpotent elements in  $\mathrm{M}_2(\F_p)$ is not closed under neither addition nor multiplication. 

For (\ref{ord}), we indeed have $R\subseteq \mathcal{O}$ by \cite[7.4]{Waterhouse}. Let $\sigma \in \mathcal O$ and observe that the coprimality of $[\mathcal O{:}R]$ to $p$ implies that there exists an integer $N$ coprime to $p$ with $N\sigma \in R$. Therefore, it suffices to prove the equivalence of inseparability and very inseparability for elements of $R$. Represent such an element $\sigma \in R$ by
$$ \sum_{i\geq 1} a_i \pi^i + \sum_{j \geq 0} b_j (\pi')^j,$$ with $\pi'=q/\pi$ and $a_i,b_i\in \Z$ (the terms containing both $\pi$ and $\pi'$ may be omitted since they do not change the image of $\sigma$ in $\End(\mathbf{D})$). Since $A$ is defined over $\F_q$ with $q=p^r$,  we have $\pi = \Fr^r$ and $\pi'=\mathrm{Ver}^r$, where $\mathrm{Ver} \colon A^{(p)} \to A$ is the  Verschiebung. On the level of Dieudonn\'e modules, $\Fr$ maps to $V$ and $\mathrm{Ver}$ maps to $F$ \cite[A.5]{Goren}, so $\sigma$ maps to the endomorphism 
$$ \tilde \sigma \coloneqq  \sum b_j F^{rj} \in \End_{K[F]}(\mathbf{D}/V\mathbf{D}). $$
In the ordinary case, the Dieudonn\'e modules of $A[p]$ and $A[p]^0$ are $$D(A[p]) = (\mathbf{E}/(V,1-F) \oplus \mathbf{E}/(F,1-V))^g$$
and $$\mathbf{D} = D(A[p]^0) = (\mathbf{E}/(V,1-F))^g$$ (since this is the subgroup scheme of $D(A[p])$ on which $V$ is nilpotent \cite[A.5]{Goren}). 
Hence $F=1$ in $\End(\mathbf{D}/V\mathbf{D}) = \mathrm{M}_g(\F_p)$, and 
$\tilde \sigma \coloneqq \sum b_j$
is a scalar multiplication; therefore, it is nilpotent if and only if it is zero (i.e., non-invertible).

The final claim follows from a result of Freeman and Lauter \cite[Prop.\ 3.7]{FreemanLauter}. \end{proof} 

 We were unable to answer the following natural questions:
 
 \begin{question}\mbox{}\begin{enumerate} \item Construct a \emph{simple} abelian variety for which very inseparable endomorphisms  do not form an ideal. \item Consider the subset of the moduli space of abelian varieties of given dimension and given degree of polarisation consisting of those abelian varieties $A$ for which inseparable isogenies are very inseparable. Is this locus dense in the moduli space? Recall that, by a result of Norman and Oort, the ordinary locus is dense \cite[Thm.\ 3.1]{NO}. 
 \end{enumerate}\end{question}

\section{The tame zeta function}

We revert to our standard assumptions and define the following general ``tame'' version of the Artin--Mazur zeta function for varieties over fields of positive characteristic (the construction is somewhat reminiscent of that of the Artin--Hasse exponential): 

\begin{definition} 
Let $K$ denote an algebraically closed field of positive characteristic $p>0$, $X/K$ an algebraic variety, and let $f \colon X \rightarrow X$ denote a confined morphism. The \emph{tame zeta function}  $\zeta^*_{f}$  is defined as the formal power series 
\begin{equation} \label{deftame} \zeta^*_{f}(z) \coloneqq \exp \left( \sum_{p \nmid n} {f_n} \frac{z^n}{n} \right), \end{equation} 
summing only over $n$ that are not divisible by $p$.
\end{definition} 

A basic observation is:
\begin{proposition} \label{p} We have identities of formal power series \begin{equation} \label{prod} \zeta_{X,f}(z) = \prod_{i \geq 0} \sqrt[p^i]{\zeta^*_{X,f^{ p^i}} (z^{p^i})}\end{equation} and 
\begin{equation} \label{inv} \zeta^*_{X,f}(z) = \zeta_{X,f}(z)/\sqrt[p]{\zeta_{X,f^{ p}}(z^p)}. \end{equation} 
\end{proposition}

\begin{proof}
For the first identity (\ref{prod}), we do a formal computation, splitting the sum over $n$ into parts where $n$ is exactly divisible by a given power $p^i$ of $p$ (denoted $p^i {\mid \! \mid} n$): 
\begin{align*}
\zeta_{X,f}(z) &=  \exp \left( \sum_{i \geq 0} \sum_{p^i {\mid \! \mid} n} \frac{f_n}{n} z^n \right) = \exp \left( \sum_{i \geq 0} \sum_{p \nmid m} \frac{f_{p^i m}}{p^i m} z^{p^i m} \right) \\
 &= \exp \left( \sum_{i \geq 0} \frac{1}{p^i}\sum_{p \nmid m} \frac{(f^{ p^i})_m}{m} \left(z^{p^i}\right)^m \right)   = \prod_{i \geq 0} \exp \left( \frac{1}{p^i}  \log \left(  {\zeta^*_{f^{ p^i}} (z^{p^i})}  \right)  \right). 
 \end{align*}
 For the second identity (\ref{inv}), we compute as follows:  \begin{align*}
\zeta^*_{X,f}(z) &=  \exp \left( \sum_{n \geq 1} \frac{f_n}{n} z^n  - \sum_{k \geq 1} \frac{f_{pk}}{pk} z^{pk}  \right) =  \exp \left( \sum_{n \geq 1} \frac{f_n}{n} z^n \right) \Big/ \exp \left(  \frac{1}{p} \sum_{k \geq 1} \frac{(f^p)_{k}}{k} z^{pk}  \right). 
\qedhere
 \end{align*}
 \end{proof}
 
 \begin{theorem} \label{tamehast} For $\sigma \acts A$, there exists an integer $t>0$ (depending on $\sigma$) such that $\left(\zeta_\sigma^*\right)^t$ is a rational function. In particular, $\zeta^*_\sigma$ is  algebraic. 
 \end{theorem}

\begin{proof} 
 Proposition \ref{rs} implies that for $p {\nmid} n$  the inseparability degree $\deg_\mathrm{i}(\sigma^n-1)=r_n$ is periodic of period $\omega$ with $r_n=r_{\gcd(n,\omega)}$. Let $\mu$ denote the M\"obius function. For $n{\mid} \omega$, define rational numbers $\alpha_n$ by  \begin{equation}\label{eqnprftame} \alpha_n =\frac{1}{n} \sum\limits_{e\mid n} \frac{\mu(n/e)}{r_e}.\end{equation} By M\"obius inversion and the equality $r_n=r_{\gcd(n,\omega)}$, we get $$\frac{1}{r_n} = \sum\limits_{d\mid \gcd(n,\omega)} d\alpha _d  \quad \mbox{for all $n\geq 1$}.$$  Therefore, 
\begin{align*} \zeta^*_\sigma(z) &= \exp \left(\sum_{p{\nmid}n} \frac{\deg(\sigma^{n} -1)}{n r_n} z^{n}\right)=\exp \left(\sum_{d{\mid}\omega}\alpha_d\sum_{p{\nmid}m} \frac{\deg(\sigma^{dm} -1)}{m} z^{dm}\right)\\ &= \prod_{d{\mid}\omega} \left( \exp\left( \sum_{p{\nmid}m} \frac{\deg(\sigma^{dm} -1)}{m} z^{dm}\right) \right)^{\alpha_d}. \end{align*} 
Using the notation of Proposition \ref{lrab}.(\ref{lrab.i}), we can rewrite this as
\begin{equation} \label{zetaD} \zeta^*_\sigma(z) =  \prod_{d{\mid}\omega} \left( D_{\sigma^d}(z^d) / \sqrt[p]{D_{\sigma^{pd}}(z^{pd}}) \right)^{\alpha_d} \end{equation} and hence the result follows from the rationality of the degree zeta functions.
\end{proof} 

The minimal exponent $t_\sigma>0$ for which $\zeta_\sigma^*(z) \in \Q(z)$ is an invariant of the dynamical system $\sigma \acts A$. We briefly discuss the arithmetic significance of such $t_\sigma$, by considering both ordinary and supersingular elliptic curves. 

\begin{proposition} \label{stranget} Let $E$ denote an elliptic curve, $\sigma \in \End(E)$, and let $t_\sigma$ be the minimal positive integer for which $\zeta_\sigma^*(z)^{t_\sigma} \in \Q(z)$. 
\begin{enumerate}
\item \label{1} If $E$ is ordinary, $t_\sigma$ is a pure $p$-th power.
\item \label{2} There exists a (supersingular) $E$ and $\sigma \acts E$ for which $t_\sigma$ is not a pure $p$-th power. 
\end{enumerate}
\end{proposition}

\begin{proof} If $\sigma$ is an endomorphism of an ordinary elliptic curve, then there is a valuation $|\cdot|$ on the quotient field $L$ of the endomorphism ring that extends the $p$-valuation and such that $\deg_{\mathrm i} \sigma = |\sigma|$ (cf.\ Remark \ref{remE}). If $\sigma$ is very inseparable, $\zeta_{\sigma}^*(z)$ is rational, and the claim is clear. Otherwise, let $s$ be the minimal positive integer for which $M\coloneqq|\sigma^s-1|<1$. We find that for integers $n$ not divisible by $p$, 
\begin{equation} \label{dio} r_n=\deg_{\mathrm{i}}(\sigma^n-1)=\left\{ \begin{array}{ll}  1 &  \mbox{ if $s {\nmid} n$, }\\  M & \mbox{ if $s \vert n$.}\end{array}\right. \end{equation} 
Substituting this into Formula \eqref{eqnprftame}, we get $\omega=s$. If $s=1$, we have $\alpha_1=1/M$, and if $s>1$, we find 
 \begin{equation} \alpha_n=\left\{ \begin{array}{ll}  1 &  \mbox{ if $n=1$,}\\  0 & \mbox{ if $n\vert s$, $1<n <s$,}\\ {(1-M)}/{(Ms)} & \mbox{ if $n=s$.}\end{array}\right. \end{equation} Since $p$ splits in $L$ \cite[\S 2.10]{Deuring}, the valuation $|\cdot|$ has residue field $\F_p$, and hence $s {\mid} (p-1)$. From Formula \eqref{zetaD}, it follows that $\zeta_\sigma^*(z)$ is a product of rational functions to powers $1/p$ and $(1-M)/(Mps)$ (and $1/(Mp)$ if $s=1$). Now with $M=p^{-r}$ for some $r \geq 1$, we find that $(1-M)/(Mps) = (p^r-1)/p^{r+1}s$, which has denominator a power of $p$, since $s$ divides $p-1$. This proves (\ref{1}). 

For (\ref{2}) consider a supersingular elliptic curve $A=E$. We have already seen in Remark \ref{remE} that the inseparability degree of an isogeny is detected by a valuation on the quaternion algebra $\End(E) \otimes \Q$, on which we now briefly elaborate. The ring $\mathcal{O} = \End (E)$ is a maximal order in a quaternion algebra, and its completion $\mathcal{O}_p=\End(E) \otimes_{\Z} \Z_p$ is an order in the unique 
quaternion division algebra $D$ over $\Q_p$ \cite{Deuring}. There exists a valuation $v\colon D \to \Z$ on $D$ with the property that $\mathcal{O}_p =\{x\in D : v(x)\geq 0\}$. Let $\mathfrak{p} =\{x\in \mathcal{O} : v(x)\geq 1\}$. Then $\mathfrak{p}$ is a two-sided maximal ideal in $\mathcal{O}$ with $p\mathcal{O}_p=\mathfrak{p}^2\mathcal{O}_p$ and we have an isomorphism $\mathcal{O}/\mathfrak{p} \simeq \F_{p^2}.$ The inseparable degree of an isogeny $\sigma \in \mathcal{O}$ is given by the formula $\deg_{\mathrm{i}}(\sigma)=p^{v(\sigma)}$, cf.\ \cite[Prop.\ 5.5]{BridyBordeaux}. 

Let $\sigma\in \mathcal{O}$ be an endomorphism such that its image in $\mathcal{O}/\mathfrak{p} \simeq \F_{p^2}$ generates the multiplicative group of the field and such that $v(\sigma^{p^2-1}-1)=1$. Then for integers $n$ not divisible by $p$ we have 
\begin{equation} \label{die} \deg_{\mathrm{i}}(\sigma^n-1)=\left\{ \begin{array}{ll}  1 &  \mbox{ if $(p^2 -1){\nmid} n$, }\\ p & \mbox{ if $(p^2 -1)\vert n$.}\end{array}\right. \end{equation} 
Let us prove that such $\sigma$  exists: choose elements $\sigma_0,\tau \in \mathcal{O}$ such that the image of $\sigma_0$ in $\mathcal{O}/\mathfrak{p}\simeq \F_{p^2}$ generates the multiplicative group of the field and  $v(\tau)=1$. Then one of the elements $\sigma_0,\sigma_0+\tau$ satisfies the desired conditions. 

Furthermore, the degree is of the form $\deg(\sigma^n-1) = m^n -\lambda^n - (\lambda')^n +1$ for $\lambda,\lambda' \in \bar \Q$ and $m\coloneqq\lambda\lambda' \in \Z$.  
Using the convenient notation $$\mathcal{Z}(z)\coloneqq\frac{\sqrt[p]{1-z^p}}{1-z},$$
a somewhat tedious computation, splitting the terms in $ \log\, \zeta^*_{\sigma}(z)$ to take into account the cases in Formula (\ref{die}), gives that 
 $$\zeta_{\sigma}^*(z)=\frac{g_{1}(z)}{\sqrt[p(p+1)]{g_{p^2-1}(z)}}, \mbox{ where }
g_{i}(z)\coloneqq\frac{\mathcal Z(z^i)\, \mathcal Z((mz)^i)}{ \mathcal Z((\lambda z)^i)\, \mathcal Z((\lambda' z)^i)}.$$
Note that $\mathcal Z(z)$ is itself a $p$-th root of a rational function.
We conclude that $t=p^2(p+1)$ suffices to have $\zeta_{\sigma}^*(z)^t \in \Q(z)$ but $\zeta_{\sigma}^*(z)^t$ is not rational for any choice of $t$ as a pure $p$-th power. 
\end{proof}

\section{Functional equations} \label{secFE}

In this section, we study the existence of functional equations for full and tame zeta functions on abelian varieties. Assume throughout the section that $\sigma$ is an isogeny. Under the transformation $z \mapsto 1/\deg(\sigma)z$, we will find a functional equation for zeta functions of very inseparable endomorphisms, and a ``Riemann surface'' version of a functional equation for the tame zeta function. Since this transformation does not make sense for $\zeta_\sigma$ as a formal power series,  $D_\sigma, \zeta_\sigma$, and $\zeta_\sigma^*$ are therefore considered as genuine functions of a complex variable, and the symbols are understood to refer to their (maximal) analytic continuations. 
\begin{proposition} \label{FED}
The degree zeta function $D_\sigma(z)$ \textup{(}cf.\ \ref{dzf}\textup{)} satisfies a functional equation of the form 
$$ D_\sigma\left(\frac{1}{\deg(\sigma)z}\right)= D_\sigma(z)
.$$

\end{proposition}

\begin{proof} We use the notations from Equation (\ref{equalityxiilambdai}). It is clear that the multiset of $\lambda_i$ is stable under the involution $\lambda \mapsto \deg(\sigma)/\lambda$. From this symmetry, we obtain a functional equation for the exponential generating function $D_{\sigma}(z)=\prod_{i=1}^r (1-\lambda_i z)^{-m_i}$ of the form $$D_{\sigma}\left(\frac{1}{\deg(\sigma)z}\right)=(-z)^{\sum_{i=1}^r m_i}\prod_{i=1}^r \lambda_i^{m_i} D_{\sigma}(z).$$ Subsituting $n=0$ into \eqref{equalityxiilambdai}  gives $\sum_{i=1}^r m_i=0$ and a direct computation using the form of $\lambda_i$ and the fact that $q$ is even shows that $\prod_{i=1}^r \lambda_i^{m_i}=1$, which gives the claim.
\end{proof}

\begin{remark}
The functional equation for $D_\sigma(z)$ can be placed in the cohomological framework from Remark \ref{remcohom}: consider the Poincar\'e duality pairing 
$ \langle \cdot, \cdot \rangle \colon \mathrm{H}^i \times \mathrm{H}^{2g-i} \otimes \Q_\ell(g) \rightarrow \Q_\ell, $
under which 
$ \langle \sigma_* x, y \rangle = \langle x, \sigma^* y \rangle, $
with $\sigma_* \sigma^* = [\deg \sigma]$. Hence if $\sigma^*$ has eigenvalues $\alpha_i$ on $\mathrm{H}^i$, then $\sigma_*$ has eigenvalues $\deg(\sigma)/\alpha_i$ on $\mathrm{H}^{2g-i}$, but these sets are the same by duality. In this way the functional equation picks up a factor $z^{\chi(A)}$, where $\chi(A)$ is the $\ell$-adic Euler characteristic of $A$. But here, $\chi(A)=0$ (since the $i$-th $\ell$-adic Betti number of an abelian variety of dimension $g$ is the binomial coefficient $\binom{2g}{i}$). 
\end{remark}

\begin{theorem} \label{thmFE} \mbox{ } 
\begin{enumerate}
\item If $\sigma$ is very inseparable, then $\zeta_\sigma(z)$ extends to a meromorphic function on the entire complex plane and satisfies a functional equation of the form 
$$ \zeta_\sigma\left(\frac{1}{\deg(\sigma)z}\right) = \zeta_\sigma(z).
$$
\item If $\sigma$ satisfies the conditions of Theorem \ref{NBsimple}, then $\zeta_\sigma(z)$ cannot satisfy a functional equation under $z \mapsto 1/\deg(\sigma)z$; actually, the intersection of the domains of $\zeta_\sigma(z)$ and $\zeta_\sigma(1/\deg(\sigma)z)$ is empty. 
\item For any confined $\sigma$, let $X_\sigma$ denote the concrete Riemann surface of the algebraic function $\zeta^*_{\sigma}(z)$ \textup{(}a finite covering of the Riemann sphere\textup{)}. Then there exists an involution $\tau \in \mathrm{Aut} (X_\sigma)$ such that the meromorphic extension $\zeta_\sigma^* \colon X_\sigma \rightarrow \widehat \C$ fits into a commutative diagram of the form 
\begin{equation} \label{fedia}
  \xymatrix{
 X_{\sigma} \ar[d]_{\zeta_\sigma^*} \ar[r]^{\tau} & X_{\sigma} \ar[d]^{\zeta_\sigma^*} \\ 
\hat \C \ar[r]^{\mathrm{id}} & \hat \C .
  }
\end{equation} 
\end{enumerate}
\end{theorem}

\begin{proof}
If $\sigma$ is very inseparable, then $\zeta_\sigma = D_\sigma$, and the result follows from Proposition \ref{FED}. 

If $\sigma$ satisfies the conditions of Theorem \ref{NBsimple} and $\zeta_\sigma$ has a natural boundary on $|z|=1/\Lambda$, then $\zeta_\sigma(z)$ and $\zeta_\sigma(\frac{1}{\deg(\sigma) z})$ are commonly defined only on $\frac{\Lambda}{\deg(\sigma)} < |z|<\frac{1}{\Lambda}$ which is empty when $\Lambda^2 \geq \deg(\sigma)$. By Proposition \ \ref{lame}.(\ref{lambdaismax}), we have $\Lambda^2 \geq \Lambda \geq \prod |\xi_i| = \deg \sigma, $ so this always holds. 

For the third part of the theorem, consider equation (\ref{zetaD}) that expresses the function $\zeta_\sigma^*$ in terms of degree zeta functions. Write $\alpha_d/p = A_d/B_d$ for coprime integers $A_d, B_d$, let $N$ denote the least common multiple of $B_d$ over all $d{\mid}\omega$ and set $\beta_d\coloneqq N\alpha_d/p \in \Z$. Then $\zeta_\sigma^*$ entends to a function on the Riemann surface $X_\sigma$ corresponding to the projective curve defined by the affine equation
$$ y^N = \prod_{d{\mid}\omega} \left( \frac{D_{\sigma^d}(x^d)^p}{D_{\sigma^{pd}}(x^{pd})} \right)^{\beta_d}$$ given by $\zeta_{\sigma}^*(x,y)=y.$ By the fact that all $D_\sigma$ satisfy the functional equation as in Proposition \ref{FED}, the map $\tau \colon X_\sigma \rightarrow X_\sigma, \tau (x,y) = \left(\frac{1}{\deg(\sigma) x},y\right) $ is an involution of $X_\sigma$ (we use that $\deg(\sigma^r)=\deg(\sigma)^r$ for any integer $r$).   The same functional equations then prove that the diagram (\ref{fedia}) commutes. 
\end{proof}

\section{Prime orbit growth} \label{orbits}

In this section, we consider the prime orbit growth for a confined endomorphism $\sigma \colon A \to A$. We are interested in possible analogues of the Prime Number Theorem (``PNT''), much like Parry and Pollicott proved for Axiom A flows \cite{PP}. In our case, it follows almost immediately from the rationality of their zeta functions that such an analogue holds for very inseparable $\sigma$. In general, however, as we will see, the prime orbit counting function displays infinitely many forms of limiting behaviour. Nevertheless, the (weaker) analogue of Chebyshev's bounds and Mertens' second theorem hold. In accordance with our philosophy, we also consider counting only ``tame'' prime orbits (i.e, of length coprime to $p$), and in this case we see finitely many forms of limiting behaviour, detectable from properties of the $p$-divisible group. Finally, we briefly discuss good main and error terms reflecting analogues of the Riemann Hypothesis. 

\begin{nd}
A \emph{prime orbit} $O$ of length $\ell=:\ell(O)$ of $\sigma\colon A \to A$ is a set $O=\{x,\sigma x, \sigma^2 x,\dots, \sigma^{\ell} x = x\} \subseteq A(K)$ of exact cardinality $\ell$. Letting $P_\ell$ denote the number of prime orbits of length $\ell$ for $\sigma$, the \emph{prime orbit counting function} is 
$ \pi_\sigma(X) \coloneqq \sum_{\ell \leq X} P_\ell. $
\end{nd}
As formal power series, the zeta function of $\sigma$ admits a product expansion
$$ \zeta_\sigma(z) = \prod_O \frac{1}{1-z^{\ell(O)}}, $$
where the product runs over all prime orbits. Since $\sigma_n = \sum_{\ell {\mid} n} \ell P_\ell$, M\"obius inversion implies that 
$P_\ell =  \frac{1}{\ell} \sum_{n \mid \ell} \mu\left(\frac{\ell}{n}\right) \sigma_n. $
Our proofs will exploit the fact that the numbers $\sigma_n$ differ from the linear recurrent sequence $\deg(\sigma^n-1)$ only by a multiplicative factor, the inseparable degree, that grows quite slowly. 

Not to complicate matters, we make the following assumption: 
\begin{center}
 \fbox{ \begin{minipage}{0.85 \textwidth} \textbf{Standing assumption/notations.} \\ The dominant root $\Lambda>1$ 
is unique.\\ The $\varpi$-periodic sequences $(r_n)$ and $(s_n)$, $s_n \leq 0$, are as in Formula \eqref{defrn}. 
\\ All asymptotic formul{\ae} in this section hold for integer values of the parameter.  \end{minipage}}
 \end{center}
By Proposition \ref{lame}.(\ref{>1}), this implies that $\Lambda>1$ is of multiplicity one. 
 We start with a basic proposition describing the asymptotics of $P_\ell$. Interestingly, the error terms are determined by the zeros of the degree zeta function. This appears to be a rather strong result with a very easy proof, dependent on the exponential growth. 
 
 \begin{proposition}\label{AsymforPl}  $\displaystyle{P_{\ell} = \frac{\Lambda^{\ell}  }{\ell r_{\ell} |\ell|_p^{s_{\ell}}} + O(\Lambda^{ \Theta \ell })}$, where 
$ \displaystyle{\Theta \coloneqq \max \{ \Re(s) : D_\sigma(\Lambda^{-s})=0 \} \in [\frac{1}{2},1)}.$ \end{proposition}

\begin{proof} From Formula \eqref{formidable}, we get $\deg(\sigma^n-1) = \Lambda^n + O(\Lambda^{ \Theta \ell })$ for 
$$\displaystyle{ \Theta \coloneqq \max_{|\lambda_i| \neq \Lambda} \frac{\log | \lambda_i|}{\log(\Lambda)}}.$$ By Proposition \ref{proplambda'}, this equals the largest real part of a zero of $D_\sigma(\Lambda^{-s})$, and $1/2 \leq \Theta < 1$. 
 Hence $$\sigma_{\ell}=\frac{\deg(\sigma^{\ell}-1)}{ \deg_{\mathrm{i}}(\sigma^{\ell}-1)} = \frac{ \Lambda^{\ell} }{ r_{\ell} |\ell|_p^{s_{\ell}}} + O(\Lambda^{ \Theta \ell }).$$
 Expressing the number of prime orbits in terms of the number of fixed points, we get $$P_{\ell}=\frac{1}{\ell} \sum_{n{\mid}\ell} \mu\left(\frac{\ell}{n}\right)\sigma_n=\frac{\sigma_{\ell}}{\ell } + \frac{1}{\ell}\sum_{\substack{n{\mid}\ell \\ n<\ell}} \mu\left(\frac{\ell}{n}\right)\sigma_n.$$ Since $|\mu(\ell/n) \sigma _n| \leq \deg(\sigma^n-1) \leq M \Lambda^n$ for some constant $M$ depending only on $\sigma$, we get $$\left|\sum_{\substack{n{\mid}\ell \\n<\ell}} \mu\left(\frac{\ell}{n}\right)\sigma_n\right| \leq \ell M\Lambda^{\ell/2}, $$ and since $\Theta \geq 1/2$, the claim follows.   
\end{proof}

 The remainder of this section is dedicated to a study of what happens to the asymptotics if we further average in $\ell$, like in the prime number theorem or Mertens' theorem. We will see that between PNT and Mertens' theorem, information about $\sigma$ being very inseparable or not gets lost. 

The next lemma is formulated in a general way and will be applied several times in order to asymptotically replace  factors ``$1/\ell$''  for $\ell \leq X$ by ``$1/X$''. This leads to simplified main terms at the cost of worse error terms (we will discuss another approach leading to a ``complicated main term with good error term'' at the end of the section). 

\begin{lemma}\label{changeltoX} Let $(a_{\ell})$ be a bounded sequence and let $\Lambda>1$ be a real number. Then $$\sum_{\ell \leq X} \frac{a_{\ell}}{\ell}\Lambda^{\ell-X} = \frac{1}{X}\sum_{\ell \leq X} a_{\ell}\Lambda^{\ell-X}+O(1/X^2).$$
\end{lemma}
\begin{proof} Write $$\sum_{\ell \leq X} \frac{a_{\ell}}{\ell}\Lambda^{\ell-X} - \frac{1}{X}\sum_{\ell \leq X} a_{\ell}\Lambda^{\ell-X}  = \sum_{\ell \leq X} \frac{a_{\ell}(X-\ell)}{X\ell} \Lambda^{\ell-X}.$$ 
With $M\coloneqq\sup |a_{\ell}| < +\infty$, the ``top half'' of this sum can be bounded as follows: $$\left| \sum_{X/2 \leq \ell \leq X} \frac{a_{\ell}(X-\ell)}{X\ell} \Lambda^{\ell-X} \right| \leq \frac{2M}{X^2} \sum_{i\geq 0} i \Lambda^{-i} = O(1/X^2)$$ while the ``bottom half'' is easily seen to be $O(X\Lambda^{-X/2})$, whence the claim. \end{proof}

\subsection*{(Non-)analogues of PNT and analogues of Chebyshev's estimates} The first application is to the following ``fluctuating'' asymptotics for the prime orbit counting function: 
\begin{proposition}\label{fluct} $ \displaystyle{\frac{X \pi_{\sigma}(X)}{\Lambda^X} = \sum_{\ell \leq X} \frac{1}{r_{\ell} |\ell|_p^{s_{\ell}}} \Lambda^{\ell-X} + O(1/X).}$\end{proposition}
\begin{proof} By Proposition \ref{AsymforPl} we see that $$\frac{X \pi_{\sigma}(X)}{\Lambda^X}  = X \sum_{\ell \leq X} P_{\ell} {\Lambda^{-X}}= X \sum_{\ell \leq X} \left( \frac{1 }{\ell r_{\ell} |\ell|_p^{s_{\ell}}}\Lambda^{\ell-X} + {\Lambda^{-X}}O(\Lambda^{\Theta \ell})\right).$$ The error terms in this sum form a geometric series and hence decrease exponentially. Applying Lemma \ref{changeltoX} to the main term, we find the stated result. 
\end{proof}

The next theorem discusses the analogue of the PNT in our setting; an analogue of Chebyshev's 1852 determination of the order of magnitude of the prime counting function 
holds in general, but the analogue of the PNT holds only for very inseparable endomorphisms. The result for general endomorphisms is similar in spirit to that for the $3$-adic doubling map considered in \cite[Thm.~3]{Stangoe}, $S$-integer dynamical systems in \cite{Everest-et-al} (from which we take the terminology ``detector group''), or to Knieper's theorem \cite[Thm.~B]{Knieper} on the asymptotics of closed geodesics on rank one manifolds of non-positive curvature. 
\begin{theorem}\label{asymptoticsforpnt} \mbox{ } 
\begin{enumerate}
\item \label{Cheb}  The order of magnitude of $\pi_\sigma(X)$ is 
$  \pi_\sigma(X) \asymp {\Lambda^{X}}/{X},$ in the sense that the function $X \pi_\sigma(X)/\Lambda^{X}$ is bounded away from $0$ and $\infty$. 
\item \label{xn} 
Consider  the ``detector'' group $$G_\sigma\coloneqq\{(a,x)\in {\Z}/{\varpi \Z} \times \Z_p : a \equiv x \bmod{ |\varpi|_p^{-1}}\}. $$ If $(X_n)$ is a sequence of integers such that $X_n \to +\infty$ and $(X_n,X_n)$ has a limit in the group $G_\sigma$, then the sequence ${X_n \pi_{\sigma}(X_n)}/{\Lambda^{X_n}}$ converges, and every accumulation point of $X \pi_\sigma(X) / \Lambda^X $ arises in this way. 
\item \label{dilim} \begin{enumerate} 
\item If $\sigma$ is very inseparable, $ \displaystyle{ \lim_{X \rightarrow +\infty} X \pi_\sigma(X) / \Lambda^X}$ exists and equals  ${ \Lambda}/{(\Lambda-1).} $
\item If $\sigma$ is not very inseparable, then the set of accumulation points of  $X \pi_\sigma(X) / \Lambda^X $ is a union of a Cantor set and finitely many points. In particular, it is uncountable. 
\end{enumerate}
\end{enumerate}
\end{theorem} 

\begin{proof}
For (\ref{Cheb}), we estimate the value of  $X \pi_{\sigma}(X)/\Lambda^X$ in terms of the sum in Proposition \ref{fluct}. The bound from above is trivial; for the bound from below we consider the terms with $\ell=X-1$ and $\ell =X$ and note that for at least one of these indices we have $|\ell|_p=1$. We thus obtain the bounds \begin{equation} \label{Chebbounds} \frac{1}{\Lambda\max(r_{\ell})} \leq \displaystyle{\liminf_{X \to +\infty}} \frac{X \pi_{\sigma}(X)}{\Lambda^X}  \leq \displaystyle{\limsup_{X \to +\infty}} \frac{X \pi_{\sigma}(X)}{\Lambda^X}  \leq \frac{\Lambda}{\Lambda-1}.\end{equation}

To prove (\ref{xn}), the formula in Proposition \ref{fluct} may be rewritten as \begin{equation}\label{paczkachusteczekeq1}\displaystyle{\frac{X \pi_{\sigma}(X)}{\Lambda^X} = \sum_{\ell = 0}^{X-1} \frac{1}{r_{X-\ell} |X-\ell|_p^{s_{X-\ell}}} \Lambda^{-\ell} + O(1/X).}\end{equation} If $(X_n)$ is as indicated, i.e., if $X_n \bmod \varpi$ stabilises (say at the value $\varpi_0 \bmod \varpi$) and $X_n$ converges to some $x$ in $\Z_p$, then individual summands in Formula \eqref{paczkachusteczekeq1} have a well-defined limit while the whole sum is bounded uniformly in $n$ by the convergent series $\sum_{t=0}^{\infty} \Lambda^{-t}$. Thus  \begin{equation}\label{paczkachusteczekeq2} \displaystyle{\lim_{n\to +\infty} \frac{X_n \pi_{\sigma}(X_n)}{\Lambda^{X_n}} = \sum_{\ell = 0}^{\infty} \frac{1}{r_{\varpi_0-\ell} |x-\ell|_p^{s_{\varpi_0-\ell}}} \Lambda^{-\ell}},\end{equation} where $(r_n)$ and $(s_n)$ are prolonged to periodic sequences for $n\in \Z$ in an obvious manner; if $x$ is a positive integer, then the term corresponding to $\ell=x$ should be construed as $\frac{\Lambda^{-\ell}}{r_{\varpi_0-\ell}}$ if $s_{\varpi_0-\ell}=0$, and $0$ otherwise.

We now prove \eqref{dilim}. When $\sigma$ is very inseparable, $\varpi=1, r_n = 1$, $s_n =0$, and Proposition \ref{fluct} implies the result by summing the geometric series $\sum_{k \geq 0} \Lambda^{-k}=1/(1-1/\Lambda)$ in (\ref{paczkachusteczekeq2}). Note that the result also follows by Tauberian methods applied to the rational zeta function $\zeta_\sigma=D_\sigma$.

In the case of general $\sigma$, we consider the map $\varphi\colon G_{\sigma}\to \R$ which associates to an element $(\varpi_0, x)\in G_{\sigma}$ the limit  $$\varphi(\varpi_0,x)=\displaystyle{\lim_{n\to +\infty} \frac{X_n \pi_{\sigma}(X_n)}{\Lambda^{X_n}}}$$ for a sequence $(X_n)$ of integers such that $X_n \to +\infty$ and $X_n$ has the limit $(\varpi_0, x)$ in $G_\sigma$. By Formula \eqref{paczkachusteczekeq2}, this map is continuous. We will show that in some neighbourhood of each point the map $\varphi$ is either constant or a homeomorphism. Note that since $G_{\sigma}$ is compact, the set of accumulation points of $X \pi_\sigma(X) / \Lambda^X $ is equal to the image of $\varphi$.

Choose $\varpi_0\bmod{\varpi}$, two distinct elements $x,y \in \Z_p$ and two sequences of integers $(X_n)$ and $(Y_n)$ which tend to infinity and such that $X_n \bmod \varpi = Y_n \bmod \varpi = \varpi_0$ and $X_n \to x$ and $Y_n \to y$ in $\Z_p$. Then by \eqref{paczkachusteczekeq2} we have \begin{equation}\label{equncmany}  \varphi(\varpi_0,x) - \varphi(\varpi_0,y)= \sum_{\ell = 0}^{\infty} a_{\ell},\end{equation} where $$a_{\ell}= \frac{1}{r_{\varpi_0-\ell}} \left(\frac{1}{|x-\ell|_p^{s_{\varpi_0-\ell}}} - \frac{1}{|y-\ell|_p^{s_{\varpi_0-\ell}}}\right) \Lambda^{-\ell}.$$

Let $k\geq 0$ be such that $|x-y|_p=p^{-k}$. The terms $a_{\ell}$ are nonzero if and only if $\ell \equiv x \pmod{p^{k+1}}$ or $\ell \equiv y \pmod{p^{k+1}}$ and furthermore $s_{\varpi_0-\ell}\neq 0$. Note that this depends only on the values of $x-\varpi_0$ and $y-\varpi_0$ modulo $\mathrm{gcd}(p^{k+1},\varpi)$.
 For $\ell$ with $a_{\ell} \neq 0$, the terms $a_{\ell}$  can be bounded from below:  $$|a_{\ell}| \geq \frac{1}{r_{\varpi_0 -\ell}} \left(p^{k s_{\varpi_0-\ell}}- p^{(k+1) s_{\varpi_0-\ell}}\right)\Lambda^{-\ell} \geq \frac{1}{2r_{\varpi_0 -\ell}} p^{k s_{\varpi_0-\ell}}\Lambda^{-\ell}$$ while clearly $|a_{\ell}| \leq \Lambda^{-\ell}$ for any $\ell$. 

We now consider two cases depending on whether or not there exists $\ell$ such that $a_{\ell}\neq 0$. 
\begin{description}
\item[\normalfont\emph{Case 1}] Assume first that there exists $\ell$ such that $a_{\ell}\neq 0$ and let $\ell_0$ be the smallest such $\ell$. Since any other such $\ell$ differs from $\ell_0$ by a multiple of $p^{k}$, we get  
 $$\left|\sum_{\ell = 0}^{\infty} a_{\ell}\right| \geq  \left(\frac{1}{2r_{\varpi_0 -\ell_0}}p^{k s_{\varpi_0-\ell_0}} - \frac{\Lambda^{-p^k}}{1-\Lambda^{-p^k}}\right)\Lambda^{-\ell_0}.$$
 Since the sequences $(r_{\ell})$ and $(s_{\ell})$ take only finitely many values, the expression on the right is positive for $k$ larger than a constant $K_0$ which depends only on $\sigma$ but not on $x$, $y$, or $\varpi_0$. Therefore from \eqref{equncmany} we conclude that if $|x-y|_p \leq p^{-K_0}$, then 
 $ \varphi(\varpi_0,x) \neq \varphi(\varpi_0,y). $
 
\item[\normalfont \emph{Case 2}] If $a_{\ell}=0$ for all $\ell$, then by Formula \eqref{equncmany}  we have $\varphi(\varpi_0,x)=\varphi(\varpi_0,y)$. Therefore the map $\varphi$ is locally constant in a neighbourhood of $(\varpi_0,x)$.
 \end{description}
 Let $p^{\nu}$ be the largest power of $p$ dividing $\varpi$. Replacing $K_0$ with $\max(K_0,\nu)$ if necessary, we see that the map $\varphi\colon G_{\sigma}\to \R$ restricted to open compact subsets $$B(\varpi_0,x)=\{(\varpi_0, Y)\in G_{\sigma} : |x-y|_p \leq p^{-K_0}\} \subseteq G_\sigma$$ is either injective (corresponding to Case 1) or constant (corresponding to Case 2). Since $G_{\sigma}$ is a disjoint union of finitely many subsets $B(\varpi_0,x)$, and since each $B(\varpi_0,x)$ is topologically a Cantor set, we conclude that the image of $\varphi$ is a union of finitely many (possibly no) Cantor sets and finitely many points.

In order to finish the proof, it is enough to note that if $\sigma$ is very inseparable, then there exists $(\varpi_0,x)\in G_{\sigma}$ for which Case 1 holds, so the image of $\varphi$ contains a Cantor set. Indeed, by Lemma \ref{unboundedinsepdeg} there exists an integer $\varpi_0$ such that $s_{\varpi_0} < 0$. It is then easy to see that  Case 1 holds for this choice of $\varpi_0$ and $x=0$.\end{proof}

\begin{example} If $\sigma$ is the (very inseparable) Frobenius  (relative to $\F_q$) on an abelian variety $A/{\F_q}$ of dimension $g$, then $\Lambda=q^g$ and we find that $\sum_{\ell \leq X} P_\ell \sim q^{g(X+1)}/(X(q^g-1))$, where $P_\ell$ is the number of closed points of $A$ with residue field $\F_{q^\ell}$.

Our warm up example from the introduction illustrates what happens in the not very inseparable case. 
\end{example}

\subsection*{Tame prime orbit counting} 
Now consider the analogous question in the tame case. 
\begin{definition}
The \emph{tame prime orbit counting function} is 
$\displaystyle{ \pi^*_\sigma(X) \coloneqq \sum_{\substack{\ell \leq X\\ p {\nmid} \ell}} P_\ell.} $
\end{definition}

\begin{remark} 
The tame zeta function $\zeta_\sigma^*(z)$ is not exactly equal to the formal Euler product over orbits of length coprime to $p$, but rather (notice the difference with Formula (\ref{prod})): $$\prod_{p {\nmid} \ell(O)} \frac{1}{1-z^{\ell(O)}} = \prod_{i \geq 0} \sqrt[p^i]{\zeta^*_\sigma(z^{p^i})}. $$  
\end{remark}

We find only finitely many possible kinds of limiting behaviour, governed by the values of the periodic sequence $(r_n)$ (the warm up example from the introduction illustrates this).

\begin{theorem} \label{PNT-tame} For any $k \in \{0,\dots, p\varpi-1\}$ the limit
\begin{equation} \label{PNTstar} \lim_{\substack{X \rightarrow +\infty \\ X \equiv k \mathrm{\, mod\, }  p\varpi}} \frac{X \pi^*_\sigma(X)}{\Lambda^{X}}  = \rho_k \end{equation} 
exists (so there is convergence along sequences of values of $X$ that converge in the ``tame detector group'' $G_\sigma^*\coloneqq{\Z}/{p\varpi}$) and is given by 
\begin{equation} \label{rhok} \rho_k = \frac{1}{\Lambda^{p\varpi}-1} \sum_{\substack{1\leq n \leq p\varpi \\ p {\nmid} n} }\frac{\Lambda^{\langle n- k  \rangle}}{r_n}, \end{equation}
where $\langle x \rangle$ denotes the representative for $x$ mod $p\varpi$ in $\{1,\dots, p\varpi\}$. 
\end{theorem}

\begin{proof}

By Proposition \ref{AsymforPl} we have $$  \pi^*_\sigma(X)=\sum_{\substack{\ell \leq X\\ p{\nmid} \ell}} \left( \frac{\Lambda^{\ell}  }{\ell r_{\ell} } + O(\Lambda^{\Theta \ell})\right).$$
The error terms in this formula form a geometric progression and hence are $O(\Lambda^{\Theta X})$.  Multiplying by $\Lambda^{-X}$ and applying Lemma \ref{changeltoX}, we get $$\frac{\pi^*_{\sigma}(X)}{\Lambda^X} =\frac{1}{X\Lambda^X} \sum_{\substack{\ell \leq X \\p{\nmid} \ell}} \Lambda^{\ell} \frac{1}{r_{\ell}} + O(1/X^2).$$ 
We split the sum by values of $r_n$, as follows: 
  \begin{align*} \lim_{X \rightarrow + \infty} \frac{X  \pi^*_\sigma(X)}{\Lambda^X} &= \lim_{X \rightarrow + \infty} \frac{1}{\Lambda^X} \left( \sum_{\substack{1\leq n \leq {p\varpi} \\ p {\nmid} n}} \frac{1}{r_n}  \sum_{s=0}^{\left \lfloor \frac{X-n}{{p\varpi}} \right \rfloor} \Lambda^{n+s{p\varpi}}  \right) \\
  &=   \lim_{X \rightarrow + \infty}  \left( \sum_{\substack{1\leq n \leq {p\varpi} \\ p {\nmid} n}}   \frac{\Lambda^{{p\varpi} \left \lfloor \frac{X-n}{{p\varpi}} \right \rfloor+{p\varpi}+n-X}}{r_n(\Lambda^{p\varpi}-1)}
\right). 
\end{align*} The limit does not converge in general, but if we put $X=Y {p\varpi} + k$ for fixed $k$ and $Y \rightarrow +\infty$, we find the indicated result, 
since ${p\varpi} \left\lfloor \frac{k-n}{{p\varpi}} \right\rfloor+{p\varpi} +n- k = \langle n-k\rangle$. 
\end{proof} 

We refer to the example in the introduction for some explicit computations and graphs.

\subsection*{Analogue of Mertens' theorem} The PNT is equivalent to the statement that the reciprocals of the primes up to $X$ sum, up to a constant, to $\log \log X + o(1/\log X)$. Mertens' second theorem is the same statement but with the weaker error term $O(1/\log X)$. It turns out that the analogue of this last theorem in our setting does hold, and very inseparable and not very inseparable endomorphisms behave in the same way. 
\begin{proposition} \label{mertens}
For some $c \in \Q$ and $c'\in \R$ we have  
$ \displaystyle{\sum_{\ell \leq X} P_\ell / \Lambda^\ell = c \log X + c'+ O(1/X).}$
\end{proposition}

\begin{proof} From Proposition \ref{AsymforPl} we find $$\sum_{\ell \leq X} P_{\ell}/\Lambda^{\ell} = \sum_{\ell \leq X}  \left( \frac{1 }{\ell r_{\ell} |\ell|_p^{s_{\ell}}} + O(\Lambda^{(\Theta-1)\ell})\right).$$ The error terms in this formula sum to $c'' + O(\Lambda^{(\Theta-1)X})$ for some $c''\in \R$ and the main terms sum to $$\sum_{j=1}^{\varpi} \frac{1}{r_j}B_{-s_j,j}(X),$$ where for integers $s\geq 0$, $\varpi>0$, and $j$, we set 
$$ B_{s,j}(X)\coloneqq\sum_{\substack{n\leq X\\ n\equiv j \mathrm{\, mod\, } \varpi}} \frac{|n|_p^s}{n}. $$
The proposition follows from 
\begin{equation} \label{bsj} B_{s,j}(X)=c_{s,j}\log X+ c'_{s,j} + O(1/X) \end{equation} for constants $c_{s,j}\in \Q$ and $c'_{s,j}\in \R$.  The case $s=0$ is well-known and we will thus limit ourselves to the case $s>0$. To prove (\ref{bsj}), we first consider the related sum 
$$A_{s,j}(X)=\sum_{\substack{n\leq X \\ n\equiv j \mathrm{\, mod\, } \varpi}} |n|_p^s$$
and we claim that \begin{equation} \label{asj} A_{s,j}(X)=c_{s,j}X+O(1) \mbox{ with }  c_{s,j}\in \Q. \end{equation} Then Abel summation gives
$$B_{s,j}(X)=\frac{A_{s,j}(X)}{X}+\int_1^X\frac{A_{s,j}(t)}{t^2} dt,$$
so (\ref{bsj}) follows, setting $c'_{s,j}=c_{s,j}+\int_1^{\infty} {(A_{s,j}(t)-c_{s,j}t)dt}/{t^2}\in \R$.
To prove (\ref{asj}), observe that the arithmetic sequence $j+ \varpi \N$ might or might not contain terms divisible by arbitrarily high power of $p$ depending on whether $|j|_p \leq |\varpi|_p$ or $|j|_p > |\varpi|_p$. In the latter case the sequence $|n|_p$ for $n\equiv j \pmod \varpi$ is constant, and the asymptotic formula for $A_{s,j}$ is clear. In the former case we write $k$ for the power of $p$ dividing $\varpi$. In the formula defining $A_{s,j}$, we isolate terms with a given value of $|n|_p$. For each integer $q\geq k$ the number of terms $n\equiv j \pmod \varpi$ with $n\leq X$ and $|n|_p=p^{-q}$ is $\frac{p-1}{p^{q-k+1}\varpi} X + O(1)$, the implicit constant being independent of $q$. We thus get the asymptotic formula $$A_{s,j}=\sum_{q\geq k} p^{-sq} \left(\frac{p-1}{p^{q-k+1}\varpi} X + O(1)\right)=c_{s,j}X+O(1)$$ with $c_{s,j} = {(p-1)p^{s(1-k)}}/{((p^{s+1}-1)\varpi)}.$ 
\end{proof}

\subsection*{Error terms in the PNT} \label{errors} We now briefly discuss how to identify good main terms and error terms in the asymptotics for the number of prime orbits. From Proposition \ref{AsymforPl}, it is immediate that 
$$ \pi_\sigma(X) = M(X) + O(\Lambda^{\Theta X })$$ with ``main term'' $$M(X)\coloneqq \sum_{\ell \leq X} \frac{\Lambda^{\ell}  }{\ell r_{\ell} |\ell|_p^{s_{\ell}}}$$ depending only on the data $(p,\Lambda, \varpi, (r_n),(s_n))$ and  the power saving in the error term is dictated by the zeros  of the degree zeta function $D_\sigma$. 

\medskip

\noindent \emph{Finding $\Theta$ geometrically.} Finding $\Theta$ can sometimes be approached geometrically, as follows. Recall that $\xi_i$ are roots of the characteristic polynomial of $\sigma$ acting on $\mathrm{H}^1$ and all $\lambda_i$ are products of such roots (corresponding to the characteristic polynomial of $\sigma$ acting on $\mathrm{H}^i = \wedge^i \mathrm{H}^1$ for various $i$). Suppose that \begin{equation} \label{conda} |\xi_i|^2=a \end{equation} for all $i$ and a fixed integer $a$. Then $\Lambda=a^g$ and $\Theta = 1-1/(2g)$, so we get an error term of the form 
$O(a^{g-1/2})$. By \cite[Chapter 4, Application 2]{Mumford}, condition (\ref{conda}) happens if for some polarisation on $A$ with Rosati involution ${ }'$, we have $\sigma \sigma'=a \mbox{ in }\End(A).$ In Weil's proof of the analogue of the Riemann hypothesis for abelian varieties $A/{\F_q}$, it is shown that this holds for $\sigma$ the $q$-Frobenius with $a=q^g$. 

\medskip

\noindent \emph{Another expression for the main term.} One may express the main term $M(X)$ as follows. For $k \in \{0,\dots,\varpi-1\}$, define 
\begin{equation} \label{fkdef} F_k(\Lambda,X) = \sum_{\substack{ \ell \leq X \\ \ell \equiv k \mathrm{\, mod\, } \varpi}}{ \Lambda^\ell/\ell}; \end{equation} 
then
\begin{equation} \label{maincont} M(X)= \sum_{k=0}^{\varpi-1} r_k^{-1} \left(F_k(\Lambda,X) + \sum_{i \geq 1}  p^{(s_k-1)i} (1-p^{-s_k}) \!\!\!\sum_{\substack{ 0 \leq k' < \varpi \\ p^i k' \equiv k\, \mathrm{mod}\, \varpi}}\!\!\! F_{k'}\left(\Lambda^{p^i},\left\lfloor \frac{X}{p^i} \right\rfloor \right) \right). \end{equation} 
We collect the information in the following proposition. 

\begin{proposition} \label{RH-prop} 
With $M(X)$ the function defined in \textup{(\ref{maincont})} using \textup{(\ref{fkdef})}, depending only on the data $(p,\Lambda, \varpi, (r_n),(s_n))$ (i.e., the growth rate $\Lambda$ and the inseparability degree pattern), we have for integer values of $X$,
$$ \pi_\sigma(X) = M(X) + O(\Lambda^{\Theta X })$$
where 
\[  \pushQED{\qed} 
\Theta =  \{ \Re(s) : s \textrm{\normalfont{ is a zero of }} D_\sigma(\Lambda^{-s}) \}. \qedhere \popQED \]
\end{proposition} 
A worked example is in the introduction.

\noindent \emph{The tame case.} In the tame setting, one similarly finds $ \pi^*_\sigma(X) {=} M^*(X)+ O( \Lambda^{\Theta X})$ with  $$ M^*(X) {=} \sum_{k=0}^{\varpi-1} r_k^{-1} \left( F_k(\Lambda,X) - \frac{1}{p} \sum_{\substack{ 0 \leq k' < \varpi \\ p k' \equiv k\, \mathrm{mod}\, \varpi}} F_{k'} \left(\Lambda^p,\left\lfloor \frac{X}{p} \right\rfloor \right) \right) . $$

\begin{remark}
Due to its exponential growth as a function of a real variable $X$, it is not possible to approximate $M(\lfloor X \rfloor)$ by a continuous function with error $O( \Lambda^{\vartheta X})$ for any $\vartheta<1$. Note that $F_k(\Lambda,X)$ can be evaluated using the Lerch transcendent. 
\end{remark}

\appendix

\section{Adelic perturbation of power series\\ \normalfont \textsc{Robert Royals and Thomas Ward}} \label{ap}

The result in this appendix comes from the thesis~\cite{Royals}
of the first author, and arose there in connection with
the following question about `adelic perturbation' of linear
recurrence sequences.
Write~$\vert m\vert_S=\prod_{\ell\in S}\vert m\vert_{\ell}$
for~$m\in\mathbf{Q}$ and~$S$ a set of primes,
and for an integer sequence~$a=(a_n)$ define a
function~$f_{a,S}$ by~$f_{a,S}(z)=\sum_{n=1}^{\infty}\vert a_n\vert_S\vert a_n\vert
z^n$.   If~$a$ is an integer linear recurrence sequence,
does~$f_{a,S}$ satisfy a P{\'o}lya--Carlson dichotomy?
That is, does~$f_{a,S}$ admit a natural boundary
whenever it does not define a rational function?
This remains open, but for certain classes of
linear recurrence and for~$\vert S\vert<\infty$,
the following theorem 
 is the key step in the argument.

\begin{theorem}\label{appendix:theorem}
Let~$a=(a_n)$ be an integer sequence with the property that
for every prime~$\ell$ there exist
constants~$n_{\ell}$ in~$\mathbf{Z}_{>0}$,~$(c_{\ell,i})_{i=0}^{n_{\ell}-1}$
in~$\mathbf{Q}^{n_{\ell}}$, and~$(e_{\ell,i})_{i=0}^{n_{\ell}-1}$
in~$\mathbf{Z}^{n_{\ell}}_{\geqslant0}$ such
that~$\vert a_n\vert_{\ell}=c_{\ell,k}\vert n\vert_{\ell}^{e_{\ell,k}}$ if ~$n\equiv k\bmod{n_{\ell}}$.
  Let~$S$ be a finite set of primes and
  write~$f(z)=\sum_{n\geqslant1}\vert a_n\vert_Sz^n$.
  If the sequence~$(|a_n|_S)$ takes infinitely many values, then~$f$ admits the unit circle as a natural boundary.
  Otherwise,~$f$ is a rational function.
\end{theorem}

The method of proof is reminiscent of Mahler's, in which functional equations allow one to conclude that certain functions have singularities along a dense set of roots of unity (compare \cite{MahlerEq}).

For the proof, it is necessary to consider a slightly more general setup. Assume that~$S$ is a finite set of
primes and for each~$\ell\in S$ there is
an associated positive integer~$e_{\ell}$, write~$e$
for the collection~$(e_{\ell})_{\ell\in S}$, and
write~$F_{S,e,r}(z)=\sum_{n\geqslant 0}\vert n-r\vert_{S,e}z^n$
for some~$r\in\mathbf{Q}$,
where~$\vert n\vert_{S,e}=\prod_{\ell\in S}\vert n\vert_{\ell}^{e_{\ell}}$.
Notice that there is always a bound of the shape$$
\frac{A}{n^B}\ll\vert n-r\vert_{\ell}\leqslant\max\{1,\vert r\vert_{\ell}\}
$$
for constants~$A,B>0$, so the radius of convergence
of~$F_{S,e,r}$ is~$1$.
If~$\vert r\vert_{\ell}>1$ for some~$\ell\in S$
then~$\vert n-r\vert_{\ell}=\vert r\vert_{\ell}$ for all~$n\in\mathbf{N}$,
and so
\[
F_{S,e,r}(z)
=
\vert r\vert_{\ell}^{e_{\ell}}
\sum_{n\geqslant 0}\vert n-r\vert_{S - \{\ell\},e}z^n
=
\vert r\vert_{\ell}^{e_{\ell}}
F_{S - \{\ell\},e,r}(z)
\]
wherever these series are defined. Thus as far as the
question of a natural boundary is concerned, we may
safely assume that~$\vert r\vert_{\ell}\leqslant1$ for all~$\ell\in S$.

Now let~$\ell \in S$ be fixed. Since~$\vert r\vert_{\ell}\leqslant1$,
we can write
\[
r=r_0+r_1\ell+r_2\ell^2+\ldots
\]
with~$r_i \in \{0,1,\ldots,\ell-1\}$
for all~$i\geqslant0$. For~$r\in\mathbf{Q}$ write~$r\bmod \ell^e$
for the positive integer~$r_0+r_1 \ell+\ldots+r_{e-1}\ell^{e-1}$.
In particular,~$r\bmod \ell^e$ is the smallest non-negative integer
with
\[
|r-(r\bmod \ell^e)|_{\ell}\leqslant \ell^{-e}.
\]
If~$n=p_1^{e_1} \cdots p_j^{e_j}$ for distinct primes~$p_i$,
then write~$r \bmod n$
for the smallest non-negative integer satisfying
\[
|r-(r\bmod n)|_{p_i}\leqslant p_i^{-e_i}
\]
for~$i=1,\dots,j$ (which exists by the Chinese remainder theorem).

Next we will obtain some functional equations
for~$F_{S,e, r}$. For~$m\geqslant 0$, we write~$t_m=
\tfrac{r- (r \bmod \ell^{m})}{\ell^m}$. Note that~$\vert t_m\vert_{p}\leqslant1$ for all~$p\in S$ and~$m\geqslant 0$. We claim that for any~$m\geqslant 1$ we have the equality
\begin{equation}
 F_{S,e,t_{m-1}}(z) =  F_{S-\{\ell\},e,t_{m-1}}(z) + \ell^{-e_{\ell}}z^{r_{m-1}}F_{S,e,t_m}(z^{\ell})
 - z^{r_{m-1}}  F_{S-\{\ell\},e,t_{m}}(z^{\ell}).\label{eq:fsr-func-rel}
\end{equation}
Indeed, we compare directly the coefficients at~$z^n$ on both sides of this equation. The coefficient on the left is~$|n-t_{m-1}|_{S,e}$. The coefficient on the right is~$|n-t_{m-1}|_{S-\{\ell\},e}$ if~$\ell{\nmid}(n-t_{m-1})$ and \[|n-t_{m-1}|_{S-\{\ell\},e} + \ell^{-e_{\ell}} \left|\frac{n-r_{m-1}}{\ell}-t_m\right|_{S,e} -\left|\frac{n-r_{m-1}}{\ell}-t_m\right|_{S-\{\ell\},e}\] otherwise. Since~$\frac{n-r_{m-1}}{\ell}-t_m=\frac{n-t_{m-1}}{\ell}$ and~$|\ell|_{S-\{\ell\},e}=1$, after an easy manipulation we see that both these coefficients are equal and hence we get~\eqref{eq:fsr-func-rel}.

Combining formul\ae~\eqref{eq:fsr-func-rel} for~$m=1,\ldots,s$, we obtain the equality 

\begin{align}\label{eq:fsr-func-rel-pre2}
 F_{S,e,r}(z)&= F_{S-\{\ell\},e,r}(z)-
  (\ell^{e_{\ell}}-1)\sum_{k=1}^{s-1}
  \frac{1}{\ell^{ke_{\ell}}}z^{r \bmod \ell^{k}}
  F_{S-\{\ell\},e,t_k}(z^{\ell^k})\nonumber \\&- \ell^{-(s-1)e_{\ell}}z^{r \bmod \ell^s}F_{S-\{\ell\},e,t_s}(z^{\ell^s}) + \ell^{-se_{\ell}}z^{r \bmod \ell^s}F_{S,e,t_s}(z^{\ell^s}) .
\end{align}

Since we have~$\vert t_s\vert_{p}\leqslant1$ for all~$p\in S$ and~$s\geqslant 0$, the coefficients in the power series~$F_{S-\{\ell\},e,t_s}(z^{\ell^s})$ and~$F_{S,e,t_s}(z^{\ell^s})$ are bounded by~$1$, and hence for~$|z|<1$ we can bound the two latter terms in \eqref{eq:fsr-func-rel-pre2} by
 \begin{align*} \left|- \ell^{-(s-1)e_{\ell}}z^{r \bmod \ell^s}F_{S-\{\ell\},e,t_s}(z^{\ell^s}) \right.&+\left. \ell^{-s e_{\ell}}z^{r \bmod \ell^s}F_{S,e,t_s}(z^{\ell^s})\right| \\  &\leqslant  (\ell^{-(s-1)e_{\ell}} +  \ell^{-se_{\ell}}) \sum_{n\geqslant 0} |z|^{n\ell^s}. \end{align*}
Thus by passing in~\eqref{eq:fsr-func-rel-pre2} with~$s$ to infinity, we obtain 

\begin{equation}\label{eq:fsr-relation2}
 F_{S,e,r}(z)= F_{S-\{\ell\},e,r}(z)-
  (\ell^{e_{\ell}}-1)\sum_{k\geqslant1}
  \frac{1}{\ell^{ke_{\ell}}}z^{r \bmod \ell^{k}}
  F_{S-\{\ell\},e,t_k}(z^{\ell^k}).
\end{equation}

\begin{lemma}\label{lem:fsr-bounded}
  Let~$S$ be a finite set of primes,~$e=\{e_{\ell}\mid\ell\in S\}$
  the associated exponents, and~$n>1$
  an integer divisible by some prime~$q\not\in S$.
  Then there is a constant~$c_{n,e,S}>0$ such that for
  any primitive~$n$th root of unity~$\mu$
  and for all~$\lambda\in[0,1)$ we have~$\vert F_{S,e,r}(\lambda\mu)|<c_{n,e,S}$.
\end{lemma}

The constant~$c_{n,e,S}$ does not depend on~$r$
under the assumption that~$|r|_{\ell}\leqslant1$ for all~$\ell\in S$.

\begin{proof}
  We proceed by induction on the
  cardinality of~$S$. For~$S=\emptyset$
  we have
  \[
  F_{S,e,r}(z)=\sum_{m\geqslant0}\vert m-r\vert_{\emptyset,e}z^m
  = \frac{1}{1-z},
  \]
and the existence of the claimed constant is clear.
  Now suppose that~$|S|\geqslant1$, let~$p\in S$ and write 
  \[
    F_{S,e,r}(z)=F_{S-\{p\},e,r}(z)-(p^{e_p}-1)
  \sum_{k\geqslant1}\frac{1}{p^{ke_p}}z^{r \bmod p^k}
  F_{S - \{p\},e,t_k}(z^{p^k}).
  \]
So,
\begin{align*}
  |F_{S,e,r}(z)| &\leqslant |F_{S - \{p\},e,r}(z)|\\
  &\qquad\qquad+(p^{e_p}-1)\sum_{k\geqslant1}\frac{1}{p^{ke_p}}|z^{r \bmod p^k}| |F_{S - \{p\},e,t_k}(z^{p^k})|\\
  &\leqslant (p^{e_p}-1)\sum_{k\geqslant0}\frac{1}{p^{ke_p}}
  |F_{S - \{p\},e,t_k}(z^{p^k})|
\end{align*}
for $|z|\leq 1$. If~$z=\lambda\mu$ for some~$\lambda\in[0,1)$ and~$\mu$ is
  a primitive~$n$th root of unity with~$q\vert n$,
  then~$z^{p^k}=\lambda'\mu'$ where~$\lambda'\in[0,1)$
    and~$\mu'$ is a primitive~$n'$th root of unity
    with~$q\vert n'$, and~$n'$ is one of finitely
    many possible values.
Thus by the inductive hypothesis
there is a constant~$c$ with~$
|F_{S - \{p\},e,t_k}(z^{p^k})| < c
$
for all~$k$,
and hence~$
|F_{S,e,r}(z)|
<
(p^{e_p}-1)c\frac{p^{e_p}}{p^{e_p}-1}.$
Taking this as~$c_{n,e,S}$ gives the lemma.
\end{proof}

\begin{lemma}\label{lemma:fsr}
Let~$S$ be a finite set of primes
and let~$r\in\mathbf{Q}$ be such that~$|r|_p\leqslant1$
for all~$p\in S$. Suppose that~$n\geqslant 1$ is an
integer divisible only by primes in~$S$, and
that~$\mu$ is a primitive~$n$th root of unity.
Writing~$n=p_1^{f_1}\cdots p_{j}^{f_j}$
where~$p_1,\ldots, p_j$ are distinct
primes in~$S$
and~$f_i\geqslant1$
for all~$i=1,\ldots, j$, we have
\[
|F_{S,e,r}(\lambda \mu)|\longrightarrow\infty
\]
as~$\lambda\to 1^-$. More precisely,
\[
\Re\bigl((-1)^j \mu^{-(r \bmod n)}F_{S,e,r}(\lambda \mu)\bigr)
\longrightarrow\infty
\]
as~$\lambda\to 1^-$ and there exists a constant~$c'_{n,e,S}$ (which does not depend on~$r$ and~$\lambda$) such that
\[
|\Im\bigl((-1)^j \mu^{-(r \bmod n)}F_{S,e,r}(\lambda \mu)\bigr)|<c'_{n,e,S}
\]
and 
\[
\Re\bigl((-1)^j \mu^{-(r \bmod n)}F_{S,e,r}(\lambda \mu)\bigr)>-c'_{n,e,S}.
\]
\end{lemma}

\begin{proof}

We again write~$z=\lambda\mu$ and define the function~$\varphi_{S,e,r,\mu}(\lambda)$ by the formula~$$ \varphi_{S,e,r,\mu}(\lambda) = (-1)^j \mu^{-(r \bmod n)} F_{S,e,r}(\lambda \mu),$$ where~$j$ is the number of prime factors of~$n$.

We proceed by induction on the number of distinct prime factors in~$n$ starting with~$n=1$.
In this case ~$\varphi_{S,e,r,\mu}(\lambda)=\sum_{m\geqslant0}|m-r|_{S,e} \lambda^m$
for each~$m$,~$\lambda^m\to1^-$ as~$\lambda \to 1^-$, and  $|m-r|_{S,e} = 1$ infinitely often. This shows that the real part tends to infinity 
as~$\lambda \to 1^-$ and is bounded from below by $0$. The imaginary part is bounded as~$F_{S,e,r}(\lambda)$
is real for all~$\lambda\in[0,1)$.

Now let~$p_1,\ldots,p_j\in S$ be distinct,
and let~$n=\prod_{i=1}^{j}p_i^{f_i}$ with~$f_i \geqslant 1$ for all~$i$.
Let~$p = p_1$ and use the variables~$r_0,r_1,\ldots$
to indicate the~$p$-adic coefficients of~$r$
and~$t_0,t_1,\dots$ to indicate the values~$
t_k=\frac{r-r \bmod p^k}{p^k}
$
for all~$k$. Assume first that~$f_1=1$. We will apply the functional equation~\eqref{eq:fsr-relation2}. For all~$k\geqslant1$,~$\mu^{p^k}$ is a primitive~$(n/p)$th root of unity and the formula~$t_k=\frac{r-r\bmod p^k}{p^k}$ implies that~$$r\bmod n \equiv r \bmod p^k + p^k (t_k \bmod (n/p)) \pmod{n}.$$
Thus Formula~\eqref{eq:fsr-relation2} after some manipulation gives  \[\varphi_{S,e,r,\mu}(\lambda)=
\varphi_{S - \{p\},e,r,\mu}(\lambda)+(p^{e_p}-1)\sum_{k=1}^\infty \frac{\lambda^{r\bmod p^k}}{p^{ke_p}}\varphi_{S - \{p\},e,t_k,\mu^{p^k}}(\lambda^{p^k}).\]
The leading term in this expression is bounded by Lemma~\ref{lem:fsr-bounded}, and the inductive hypothesis applied to the terms~$\varphi_{S - \{p\},e,r,\mu^{p^k}}(\lambda^{p^k})$ shows that their real part tends to~$+\infty$ as~$\lambda\to1^-$ and is bounded away from~$-\infty$ independently of~$r$ and~$\lambda$. Since 
these terms appear within the geometric progression~$\sum_{k=1}^\infty p^{-{ke_p}}$, we obtain that~$$\varphi_{S,e,r,\mu}(\lambda) \rightarrow \infty$$  as~$\lambda\to1^-$ and the same argument proves the latter claim. This proves the inductive step for the case~$f_1 = 1$.

We will use this as the base case for a
second inductive proof for~$f_1>1$. The argument in this case is similar except that we will use the functional equation~\eqref{eq:fsr-func-rel} instead of~\eqref{eq:fsr-relation2}. As before,~$\mu^p$ is a primitive~$(n/p)$th root of unity and~$$r\bmod n \equiv r \bmod p+ p (t_1 \bmod (n/p)) \pmod{n}.$$  Thus Formula~\eqref{eq:fsr-func-rel} after some manipulation gives  
\[\varphi_{S,e,r,\mu}(\lambda)=
\varphi_{S - \{p\},e,r,\mu}(\lambda)+p^{-e_p}\lambda^{r\bmod p}\varphi_{S,e,t_1,\mu^{p}}(\lambda^{p})-\lambda^{r\bmod p}\varphi_{S - \{p\},e,t_1,\mu^p}(\lambda^{p}).\]
The first and the third terms in this expression are bounded by Lemma~\ref{lem:fsr-bounded}, and hence the claim follows immediately from the inductive hypothesis applied to the term~$\varphi_{S,e,t_1,\mu^{p}}(\lambda^{p})$. This concludes the induction.\end{proof}

\begin{proof}[Proof of Theorem \ref{appendix:theorem}]
  If~$c_{\ell,k}=0$ for some~$\ell\in S$ and~$k$
  we will automatically take~$e_{\ell,k}=0$ as the
  power of~$\vert n\vert_{\ell}$ plays no role.
  Another case we wish to avoid is if for some~$\ell$
  and~$k\in \{0,1,\ldots,n_{\ell}-1\}$,
  the value~$\vert n\vert_{\ell}$ is constant for
  all~$n\equiv k \bmod n_{\ell}$.
  Writing~$v_{\ell}$ for the~$\ell$-adic order,
  this happens exactly when~$v_{\ell}(n_{\ell})>v_{\ell}(k)$,
  and in this case~$\vert n\vert_{\ell}=\vert k\vert_{\ell}$.
  If this is the case and~$e_{\ell,k}\neq0$,
  then we will set~$e_{\ell,k} = 0$ and
  substitute~$c_{\ell,k}\vert k\vert_{\ell}^{e_{\ell,k}}$ for~$c_{\ell,k}$.
Let~$N={\rm{lcm}}\{n_p \,\mid\, p\in S\}$.
For each~$j\in\{0,1,\dots, N-1\}$
consider the value of~$|a_n|_S$ when~$n\equiv j\bmod N$.
For each~$p$,~$n\equiv j\bmod N$
and thus~$n\equiv j\bmod n_p$ as~$n_p\vert N$.
Let~$k_{p,j}$ be the unique element of~$\{0,1,\ldots,n_p-1\}$
such that~$k_{p,j}\equiv j\bmod n_p$. So
\[
|a_n|_S
=
\prod_{p \in S} |a_n|_p
=
\prod_{p \in S} c_{p,k_{p,j}} |n|_p^{e_{p,k_{p,j}}}
\]
as~$n \equiv j \equiv k_{p,j} \bmod n_p$ for all~$p \in S$.
If for any nonzero~$n$ with~$n \equiv j \bmod N$
we have~$|a_n|_S = 0$, or equivalently~$a_n = 0$, we define~$S_j = \emptyset$
and~$d_j = 0.$
If this is the case, then it follows that for this value~$n$ 
\[
0=\prod_{p \in S} c_{p,k_{p,j}}|n|_p^{e_{p,k_{p,j}}}
\]
and~$|n|_p^{e_{p,k_{p,j}}} \neq 0$ implies that~$c_{p,k_{p,j}} = 0$ for some~$p \in S$.
This in turn implies that~$|a_m|_S = 0$ and hence~$a_m = 0$ for any~$m \equiv j \bmod N$.
If, on the other hand,
for some~$n \equiv j \bmod N$ we have~$|a_n|_S \neq 0$
then for all~$m \equiv j \bmod N$
we have~$|a_m|_S \neq 0$ and hence~$c_{p,k_{p,j}} \neq 0$ for all~$p \in S$.
If for a prime~$p \in S$ we
have~$v_p(N)>v_p(j)$, then for all~$n \equiv j \bmod N$ we have~$ |n|_p = |j|_p$.
We will split~$S$ into the disjoint union~$S_j \sqcup S_j' \sqcup S_j''$, where
\[
S_j = \{p \in S\mid v_p(N) \leqslant v_p(j) \mbox{ and } e_{p,k_{p,j}} \neq 0\},
\]
\[
S_j' = \{p \in S ~\mid~ v_p(N) > v_p(j) \mbox{ and } e_{p,k_{p,j}} \neq 0\},
\]
and
\[
S_j'' = \{p \in S ~\mid~ v_p(N) > v_p(j) \text{ and } e_{p,k_{p,j}} = 0\}.
\]
Thus for all~$n \equiv j \bmod N$ we have
\[
|a_n|_S = \prod_{p \in S} c_{p,k_{p,j}}\cdot \prod_{p \in S_j'}|j|_p^{e_{p,k_{p,j}}} \cdot |n|_{S_j,e^{(j)}},
\]
where~$e^{(j)}$ denotes the collection of
exponents~$\{e_{p,k_j}\mid p \in S_j\}$.
Set
\[
d_j = \prod_{p \in S} c_{p,k_{p,j}}\cdot \prod_{p \in S_j'}|j|_p^{e_{p,k_{p,j}}}
\]
and~$|a_n|_S = d_j |n|_{S_j,e^{(j)}} \mbox{ for all } n \equiv j \bmod N$.

Assume that the sequence~$(|a_n|_S)$ takes infinitely many values. This implies that there exists some~$j$ for
which~$S_j$ is non-empty. By our assumption, for such~$j$ we have~$d_j\neq 0$. Consider the family of sets~$\{S_j \mid 0 \leqslant j < N\}$,
partially ordered by inclusion.
Since it  is finite and the~$S_j$ are
not all empty, there is
a non-empty maximal element~$S_{j_0}$.
Write
\[
f(z)
=
\sum_{n=1}^\infty |a_n|_S z^n
=
\sum_{j=0}^{N-1} \sum_{n \equiv j \, (N)} |a_n|_S z^n
=
\sum_{j=0}^{N-1} f_j(z)
\]
where
\begin{align*}
f_j(z) &= \sum_{n \equiv j  \, (N)} |a_n|_S z^n
=
\sum_{n \equiv j  \, (N)} d_j |n|_{S_j,e^{(j)}} z^n
=
\sum_{k = 0}^\infty d_j |kN + j|_{S_j, e^{(j)}} z^{kN + j}\\
&=
d_j |N|_{S_j,e^{(j)}}\sum_{k = 0}^\infty|k + j/N|_{S_j,e^{(j)}} z^{kN + j}
=
d_j |N|_{S_j,e^{(j)}} z^j g_j(z^N)
\end{align*}
with~$g_j(z)=F_{S_j,e^{(j)},-j/N}(z)$. Thus~$f=h_1+h_2$, where~$h_1$ is the sum of the~$f_j$
with~$S_j = S_{j_0}$ and~$h_2$ is the sum of the~$f_j$ with~$S_j \neq S_{j_0}$.
Let~$n=\prod_{q \in S_{j_0}} q^{f_q}$
be an integer divisible by every prime in~$S_{j_0}$ and by no other primes
such that for each~$q \in S_{j_0}$
we have~$f_q > v_q(N)$
and let~$\mu$ be a primitive~$n$th root of unity.
If~$j$ with~$0\leqslant j<N$
has~$S_j\neq S_{j_0}$ then~$
f_j(\lambda \mu)
=
d_j|N|_{S_j,e^{(j)}}(\lambda\mu)^j g_j(\lambda^N \mu^N)
$
is bounded as~$\lambda\to1^-$ by Lemma~\ref{lem:fsr-bounded}
as~$\mu^N$ is an~$\frac{n}{N}$th root of unity and~$\frac{n}{N}$
is divisible by every prime in~$S_{j_0}$ and hence by some prime
not in~$S_j$ by maximality of~$S_{j_0}$. Thus~$|h_2(\lambda \mu)|$
is bounded as~$\lambda\to1^-$.
Suppose instead that~$S_j = S_{j_0}$. By Lemma~\ref{lemma:fsr} we have that
\[
\Re\bigl((-1)^m (\mu^N)^{-(-j/N \bmod n/N)}g_j(z^N)\bigr)
\longrightarrow\infty
\]
as~$\lambda\to1^-$
where~$m=|S_{j_0}|$.
Equivalently,
\[
\Re\bigl((-1)^m \mu^{(j \bmod n)}g_j(z^N)\bigr)
\longrightarrow\infty,
\]
and thus
\begin{align*}
\Re\bigl((-1)^m z^j g_j(z^N)\bigr)
\longrightarrow\infty
\end{align*}
as~$\lambda\to1^-$.
As the real part of every term in~$h_1(z)$ goes to~$\infty$, this means that
\[
\Re\bigl((-1)^m f(\lambda \mu)\bigr)\longrightarrow\infty
\]
as~$\lambda\to1^-$.
Since this is true for any~$\mu$
that is a~$(\prod\limits_{q \in S_{j_0}} q^{f_q})$th root of unity with each~$f_q > v_q(N)$,
these singularities form a dense set on the unit circle.
It follows that~$f$
admits a natural boundary on the unit circle.

For the second part of the theorem,
assume that the sequence~$(|a_n|_S)$ takes only finitely many values. Then~$(|a_n|_S)$ is periodic modulo~$N$, and thus
\[
f(z)
=
\sum_{j=1}^{N} \sum_{n \equiv j\, (N)} |a_j|_S z^n
=
\sum_{j=1}^{N}|a_j|_S \sum_{m=0}^\infty z^{mN + j}
=
\sum_{j=1}^{N}|a_j|_S \frac{z^j}{1-z^N},
\]
completing the proof.
\end{proof}

\medskip

\footnotesize{ 
\indent Appendix author information: 

\medskip

\indent Ziff Building, University of Leeds, Leeds LS2 9JT, UK \\
\indent \emph{E-mail address:} {\tt t.b.ward@leeds.ac.uk} 

\medskip

\indent  School of Mathematics, University of East Anglia, Norwich NR4 7TJ, UK \\
\indent \emph{E-mail address:} {\tt aradesh@gmail.com}
}

\addtocontents{toc}{\SkipTocEntry}
\bibliographystyle{amsplain}

\end{document}